%% file: Ross_PSPUM.tex
\documentclass[11pt]{amsart}
\usepackage{amsmath}
\usepackage{amsthm}
\usepackage{amssymb}
\usepackage{latexsym}
\usepackage[english]{babel}
\usepackage{graphicx}
\usepackage{color}
\usepackage{comment}
\newtheorem{theorem}{Theorem}[section]
\newtheorem{lemma}[theorem]{Lemma}
\newtheorem{proposition}[theorem]{Proposition}

\theoremstyle{definition}
\newtheorem{example}[theorem]{Example}
\newtheorem{definition}[theorem]{Definition}
\newtheorem{corollary}[theorem]{Corollary}
\newtheorem{remark}[theorem]{Remark}

\newcommand{\Sh}{\operatorname{Sh}}
\newcommand{\MA}{\operatorname{MA}}

\newcommand{\Psh}{\operatorname{Psh}}
\renewcommand{\limsup}{\operatorname{limsup}}
\numberwithin{equation}{section}

\begin{document}

\title[The Dirichlet Problem for the Complex HMAE]{The Dirichlet Problem for the Complex Homogeneous
    Monge-Amp\`ere Equation}


%    author one information
\author{Julius Ross}
\address{DPMMS\\University of Cambridge\\UK}
\email{j.ross@dpmms.cam.ac.uk}

%    author two information
\author{David Witt Nystr\"om}
\address{Department of Mathematical Sciences\\Chalmers University of Technology and the University of Gothenburg\\Sweden}
\email{wittnyst@chalmers.se}
\email{danspolitik@gmail.com}
\dedicatory{Dedicated to Sir Simon Donaldson on the occasion of
  his 60th birthday}

%\subjclass[2000]{Primary }
%    The 2010 edition of the Mathematics Subject Classification is
%    now available.  If you are citing a classification from the
%    new scheme, use the following input coding instead.
%\subjclass[2010]{Primary }

\date{}

\begin{abstract}
    We survey the Dirichlet problem for the complex Homogeneous Monge-Amp\`ere Equation, both in the
    case of domains in $\mathbb C^n$ and the case of compact K\"ahler manifolds parametrized by a
    Riemann surface with boundary.  We then give a self-contained account of previous work of the
    authors that connects this with the Hele-Shaw flow, and give several concrete examples
    illustrating various phenomena that solutions to this problem can display.
\end{abstract}

\maketitle


\section{Introduction}\label{sec:introduction}

Let $X$ be a compact complex manifold of dimension $n$,  $\Sigma$ be a compact Riemann-surface with non-empty smooth boundary, and $\{\omega_{\tau}\}_{\tau\in \partial \Sigma}$  be a  family of real $(1,1)$ forms or currents on $X$.   Usually we will assume $\omega_{\tau}$ is smooth, strictly positive, and varies smoothly in $\tau$, thus giving a smooth family of K\"ahler forms parameterized by $\partial\Sigma$.    

The Dirichlet problem for the complex Homogeneous Monge-Amp\`ere Equation (HMAE) in this setting seeks a real $(1,1)$ form, or current, $\Omega$ on $X\times \Sigma$ satisfying
\begin{align}
\Omega^{n+1} &=0,\label{eq:chmaeintro}\tag{$\diamondsuit$}\\
\Omega &\ge 0,\nonumber\\
\Omega|_{X\times \{\tau\}} &= \omega_\tau \text{ for } \tau \in \partial \Sigma.\nonumber
\end{align}
It is known, under suitable hypothesis on $\Sigma$, that one can always find a solution to this equation in the sense of currents, where $\Omega^{n+1}$ is to be understood as the Monge-Amp\`ere operator defined by Bedford-Taylor.  Following Donaldson, we say that a solution is \emph{regular} if $\Omega$ is smooth and the $(1,1)$-forms $\Omega|_{X\times \{ \tau \}}$  are strictly positive for all $\tau\in \Sigma$.  Thus a regular solution gives a family of K\"ahler forms on $X$ parameterized by $\Sigma$.     The guiding question we will be interested in is how far a general solution to \eqref{eq:chmaeintro} is from being regular, and whether there are conditions under which a regular solution can be guaranteed.

The first part of this survey considers various instances of the HMAE, beginning with the work of Bedford-Taylor on the Dirichlet problem for domains in $\mathbb C^n$ and the pluricomplex Green function introduced by Klimek.  We then turn to the setting above,  which took on particular importance through work of Semmes and independently Donaldson, who observed that it comes up naturally as the geodesic equation in the space of K\"ahler metrics on $X$.

In the second part we give a self-contained account of previous work of the authors that relates the
HMAE when $X$ itself is a Riemann-surface to a well-known problem in fluid mechanics called the
Hele-Shaw flow.  In doing so we are able to much better understand this regularity problem, and we
end with four concrete examples that show the kind of irregular behaviour that solutions to
\eqref{eq:chmaeintro} can display.   In each case these will be obtained by constructing Hele-Shaw flows with particular properties.

Our first example (Section \ref{sec:slits}) considers flows developing a ``self-tangency" in which at a particular time a Hele-Shaw domain is simply connected, but has boundary that self-intersects tangentially at a  point.  From this we can produce a solution to the HMAE that is not twice differentiable at certain points.  In fact the relation between the self-tangency and this non-differentiability is extremely explicit, and one can not only see at which points this occurs but also the directions along which the second derivative does not exist.  

The second example (Section \ref{sec:multiply}) shows the Hele-Shaw flow becoming non-simply connected, from which we can produce a solution to the HMAE admitting a definite region that does not intersect any proper harmonic disc.  This obstructs the existence of a local Monge-Amp\`ere foliation with proper leaves, and so this weak solution is far away from being regular. 

In the third example (Section \ref{sec:ex3}) we produce a flow whose final domain is simply connected but has as boundary some (non-trivial) curve.  From this we get examples of solutions to the HMAE that fail to to have the so-called ``maximal rank" property.  In the final example (Section \ref{sec:acute}) we apply work of Sakai concerning the Hele-Shaw flow for domains that have acute corners to obtain boundary data for the HMAE over the disc that is $\mathcal{C}^{1,\alpha}$ for all $\alpha<1$ but whose weak solution is not even $\mathcal{C}^1$ up to the boundary.

\section{Preliminaries}

%\subsection{Subharmonic and Plurisubharmonic Functions}
Throughout, $\mathbb D$ and $\overline{\mathbb D}$ denote the open and closed unit disc in $\mathbb C$ respectively, and $\mathbb D^\times$ and $\overline{\mathbb D}^\times$ will denote these with the origin removed.   On any complex manifold $X$ we use the convention
$$d^c = \frac{i}{2\pi}(\overline{\partial} - \partial), \text{ so } dd^c = \frac{i}{\pi} \partial \overline{\partial}.$$
%Given any $\phi\in C^{\infty}(X)$ we set
%$$ \omega_{\phi} := \omega + dd^c\phi$$
%and let
%$$\mathcal K(X,\omega) := \{ \phi\in C^{\infty}(X) : \omega_\phi \text{ is strictly positive}\}$$
%which is the space of \emph{K\"ahler potentials} associated to $(X,\omega)$.    By an application of Stokes' Theorem, the area of $X$ is the same if taken with respect to $\omega_\phi$ for any $\phi\in \mathcal K(X,\omega)$.  Since our results will be essentially unchanged by scaling $\omega$ by a positive real number, we normalise and assume that
%$$ \int_X \omega =1.$$
%The same notation is used for potentials that are not necessarily smooth or even bounded.  

Given a closed real $(1,1)$-form $\theta$ on a connected $X$, we say $u:X\to [-\infty,\infty)$ is \emph{$\theta$-pluri\-subharmonic} (or simply \emph{$\theta$-psh}) if whenever locally $\theta = dd^cv$ then $u+v$ is plurisubharmonic.  When $u$ is upper-semicontinuous, locally integrable, and not identically $-\infty$ we write
$$\theta_u:=\theta+dd^c u,$$ 
and if $u$ is $\theta$-psh then $\theta_u\ge 0$ in the sense of currents.   The space of plurisubharmonic functions on $X$ is denoted by $\Psh(X)$ and the $\theta$-psh functions by $\Psh(X,\theta)$.    When $X$ has complex dimension $1$, being plurisubharmonic is the same as being subharmonic, and we use the more common notation $\Sh(X,\theta)$ for the space of $\theta$-subharmonic functions in this case.

For general $\theta$ there might not be any $\theta$-psh functions, but if $\theta$ is strictly positive (and thus a K\"ahler form) then there certainly are (for instance the constant functions).  We say $u$ is a \emph{K\"ahler potential} (with respect to $\theta$) if it is smooth and $\theta_u$ strictly positive, and denote the space of such potentials by $\mathcal K(X,\theta)$.

\begin{comment}
For an upper-semicontinuous function $\psi:X\to \mathbb R\cup \{-\infty\}$ that is not identically $-\infty$ and is in $L^1_{loc}(X)$ one can define $dd^c\psi$ as a current, and we let $$ \omega_{\psi}: = \omega + dd^c\psi.$$   Given any smooth $(1,1)$-form $\omega'$ set
\begin{align*}
\Sh(X,\omega') &= \{ \psi:X\to \mathbb R\cup \{-\infty\} : \psi \text{ is } \omega'\text{-subharmonic}\}\\
& = \{  \psi:X\to \mathbb R\cup \{-\infty\}  : \psi \text{ is upper-semicontinuous  and } \omega'_\psi\ge 0 \} \cup \{-\infty\}
\end{align*}
where $\omega'_{\psi}\ge 0$ is to be understood as holding in the sense of currents.  We remark that in the second displayed equation we are implicitly assuming $\psi\in L^1_{loc}(X)$ which is harmless as this is implied by $\psi$ being $\omega'$-subharmonic \cite[Corollary 2.4.7]{Klimek}. 
\end{comment}

%\subsection{Complex Monge-Amp\`ere Measures}

Suppose $u$ is plurisubharmonic on a domain in $\mathbb C^n$.  When $u$ is twice differentiable its Monge-Amp\`ere measure is defined as $\MA(u):=(dd^cu)^n$. In general, $dd^cu$ is merely a positive current so the wedge product $(dd^cu)^n$ does not immediately make sense.  However, Bedford-Taylor showed that the complex Monge-Amp\`ere operator can be defined for plurisubharmonic functions that are locally bounded \cite{BedfordTaylor2}. The idea is to define the current $(dd^cu)^m:=dd^c(u(dd^cu)^{m-1})$ inductively for $m\in\{1,2,...,n\}$. Assuming $(dd^cu)^{m-1}$ is a positive current, it follows that it has measure coefficients, and since $u$ is locally bounded $u(dd^cu)^{m-1}$ will also be a current with measure coefficients.   Thus $(dd^cu)^m:=dd^c(u(dd^cu)^{m-1})$ is a well defined current and Bedford-Taylor prove that it is positive. Hence by induction $\MA(u):=(dd^c u)^n$ is a well defined positive $(n,n)$-current, i.e. a positive measure.

When $u$ is locally bounded its Monge-Amp\`ere measure relative to a smooth $(1,1)$-form $\theta$ is defined locally where $\theta=dd^c v$ as $\MA_{\theta}(u):=MA(u+v)$. 

\section{The HMAE on domains in $\mathbb{C}^n$} \label{Sect:Dir}

\subsection{Perron-Bremermann Envelopes}

Let $U$ be a smoothly bounded domain in $\mathbb{C}^n$ and $\phi\in \mathcal{C}^0(\partial U)$. The Dirichlet problem for the complex Homogeneous Monge-Amp\`ere Equation (HMAE) on $U$ with boundary data $\phi$ asks for a locally bounded $u\in \Psh(U)$ such that
\begin{align}\label{eq:hmaelocal}
    \MA(u)&=0 \text{ on }U,\\
    \lim_{z\to\zeta}u(z)&=\phi(\zeta) \text{ for all } \zeta\in \partial U.\nonumber
\end{align}
As in the one dimensional case (i.e.\ when solving the Laplace equation) solutions to HMAE can be found using envelope constructions.  The \emph{Perron-Bremermann} envelope $u$ is defined as 
\begin{equation}\label{eq:perron}
u:={\sup}^*\{v\in \Psh(U): \limsup_{z\to\zeta}v(z)\leq \phi(\zeta), \forall \zeta\in \partial U\},
\end{equation}
where ${\sup}^*$ means the upper-semicontinuous regularization of the supremum.  A proof of the following statement can be found in \cite[p18]{Guedjbook}.  

\begin{theorem} \label{thm:locdir}
Assume $U$ is a smoothly bounded and strictly pseudoconvex domain in $\mathbb C^n$.  Then the Perron-Bremermann envelope $u$ is the unique solution to the Dirichlet problem for the HMAE \eqref{eq:hmaelocal} with boundary $\phi$.
%\begin{align*}
%MA(u)&=0,\\  
%\lim_{z\to\zeta}u(z)&=\phi(\zeta).
%\end{align*}
\end{theorem}

 One can similarly consider the inhomogeneous Monge-Amp\`ere Equation, in which one seeks a solution to $MA(u)= dV$ where $dV$ is a given smooth volume form.  Through the work of Caffarelli, Kohn, Nirenberg, Spruck \cite{Caf1,Caf2,Caf3} it is known that, as long as $U$ is strictly pseudoconvex, if $\phi$ is smooth then the solution to the inhomogeneous problem with boundary data $\phi$ is also smooth.  However for the homogeneous case that we are interested in the answer is more subtle.  %The following statement is due to Krylov.

%We recall that through the work of Yau(???) that a solution to the non-degenerate Monge-Amp\`ere equation $MA(u)=dV$ where $dV$ is some volume form is automatically $\mathcal{C}^{\infty}$ if the boundary data is $\mathcal{C}^{\infty}$ (see e.g. \cite[Ch. 7]{Guedjbook}).  For the degenerate case the answer is more subtle.

\begin{theorem}[Krylov] \label{thm:locreg}
Assume $U$ is a smoothly bounded and strictly pseudoconvex domain in $\mathbb C^n$.  If $\phi\in \mathcal{C}^{\infty}(\partial U)$ then the solution $u$ to the HMAE with boundary data $\phi$ lies in $\mathcal{C}^{1,1}(\overline{U})$.
\end{theorem}

The next example shows that this regularity result is optimal.

\begin{example} \label{ex:gamelin}
Let $U:=\mathbb{B}$ be the unit ball in $\mathbb{C}^2$ and for $(z,w)\in \partial \mathbb{B}$ let $\phi(z,w):=(|z|^2-1/2)^2=(|w|^2-1/2)^2$. It can then be easily checked that the solution to the Dirichlet problem is given by $$u(z,w)=(\max\{0, |z|^2-1/2, |w|^2-1/2\})^2.$$ This clearly fails to be $\mathcal{C}^2$ along the hypersurfaces $|z|^2=1/2$ and $|w|^2=1/2$.
\end{example}

\subsection{Regular solutions, Monge-Amp\`ere foliations and proper harmonic curves}

\begin{definition}
We say that a solution $u$ to Dirichlet problem \eqref{eq:hmaelocal} is \emph{regular} if $u\in\mathcal{C}^{\infty}(\overline{U})$ and if at every point of $U$ the kernel of $dd^cu$ has complex dimension 1.
\end{definition}

\begin{definition}
Let $u$ be a solution to \eqref{eq:hmaelocal}.   A subset of $U$ is called a \emph{proper harmonic curve} of  $u$ if it is the image of a proper holomorphic map $f:\Sigma\to U$ from a Riemann-surface $\Sigma$ such that $u\circ f$ is harmonic on $\Sigma$. 
\end{definition}

If $u$ is regular, the kernel of $dd^cu$ defines a one-dimensional distribution which turns out to be integrable, and so by Frobenius Integrability Theorem yields a foliation of $U$ whose leaves are proper harmonic curves.   This is known as the associated \emph{Monge-Amp\`ere foliation}.  Thus  a way to gauge the failure of regularity of a solution is to see how far the set of proper harmonic curves is from foliating the domain. 

Returning to Example \ref{ex:gamelin} one easily checks that the set of proper harmonic curves consists of the discs $\{(z,c)\in \mathbb{B}\}$ and $\{(c,w)\in \mathbb{B}\}$ for $1/2\leq |c|^2<1$ together with the discs $\{(z,cz)\in \mathbb{B}\}$ for $|c|=1$. Interestingly, even though this set of discs is far from foliating the domain $\mathbb{B}$ their boundary circles do foliate $\partial \mathbb{B}$. In particular $\partial \mathbb{B}$ is contained in the closure of the union of proper harmonic curves. We are not aware of any examples of solutions to the above Dirichlet problem where this is not the case.

A related but different issue is that of finding local harmonic discs, i.e. non-trivial but not necessarily proper holomorphic discs along which $u$ is harmonic. Indeed looking at Example \ref{ex:gamelin} it is clear that through any point in $\mathbb{B}$ passes at least one local harmonic disc. However, an interesting construction of Sibony shows that this not always has to be the case (for the details see \cite[Sect. 3.5.1]{Guedjbook} and references therein).   

\subsection{Pluricomplex Green functions} \label{sect:pluricomplex}

Another manifestation of the HMAE comes through the so-called pluricomplex Green function.  Let $U$ be a smoothly bounded strictly pseudoconvex domain in $\mathbb{C}^n$ and fix a point $z_0\in U$. 

\begin{definition}
The \emph{pluricomplex Green function} of $U$ with singularity at $z_0$ is defined as $$u_{U,z_0}:={\sup}^*\{v\in \Psh(U): v\leq 0, \nu_{z_0}(v)\geq 1\}.$$
\end{definition}
Above $\nu_{z_0}(v)$ denotes the Lelong number of $v$ at the point $z_0$, defined by
$$ \nu_{z_0}(v) = \sup\{ t : v \le t \ln |z-z_0|^2 + O(1)\}$$
(we refer the reader to \S \ref{sec:lelong} for more on Lelong numbers).

\begin{theorem}[Demailly, B\l ocki] \label{thm:pluricomplex}
The pluricomplex Green function $u_{U,z_0}$ solves the HMAE on $U\setminus \{z_0\}$ and is $\mathcal{C}^{1,1}$ on $\overline{U}\setminus \{z_0\}$.
\end{theorem}

We say that the pluricomplex Green function is \emph{regular} if it is $\mathcal{C}^{\infty}$ on $\overline{U}\setminus \{z_0\}$ and the kernel of $dd^cu$ one-dimensional on $\overline{U}\setminus \{z_0\}$. Given a regular solution, the corresponding Monge-Amp\`ere foliation will then consist of holomorphic curves attaching to $\partial U$ and by the maximum principle all will pass through the point $z_0$. 

It was shown by Lempert \cite{Lempert} that the pluricomplex Green function is regular when the domain $U$ is smoothly bounded and strictly convex. The discs of the corresponding Monge-Amp\`ere foliation contain important information of the domain $U$. 

For $z,w\in U$, the Kobayashi distance between $z$ and $w$, denoted by $\delta_K(z,w)$, is defined as the infimum of the Poincar\'e distance between pairs $x$ and $y$ in $\mathbb{D}$ over all holomorphic discs $f:\mathbb{D}\to U$ with $f(x)=z$ and $f(y)=w$. Such a disc connecting $z$ and $w$ is called \emph{extremal} if it realizes the Kobayashi distance. Lempert proves in \cite{Lempert} that when $U$ is smoothly bounded and strictly convex, for any $z,w\in U$ there exists an extremal disc (unique up to reparametrization) and that this is exactly the disc of the Monge-Amp\`ere foliation of $u_{U,z_0}$ that passes through $w$. We will discuss more of Lempert's proof in connection with the Donaldson Openness Theorem in \S \ref{sect:donopthm}.

In contrast, Bedford-Demailly \cite{BedfordDemailly} give examples of smoothly bounded strictly pseudoconvex domains with a pluricomplex Green function which is not $\mathcal{C}^2$ up to the boundary. It seems not to be known whether they also fail to be $\mathcal{C}^2$ in the interior of the domain.

\subsection{Bibliographical Remarks}

%There are several accounts of the basic potential theory we have used, for example \cite{Klimek}.    The Lelong number originated in work of Lelong \cite{Lelong57} (see also \cite{Demaillytrento} for a modern treatment).    

The reader interested in more comprehensive surveys on this topic is referred to Berndtsson \cite{Berndtsson2}, Guedj \cite{Guedjbook}, Guan \cite{Guansurvey}, Ko\l odziej \cite{Kol1,Kol2}, Phong-Song-Sturm \cite{PhongSongSturmSurvey} and Zeriahi \cite{Zeriahisurvey}.   The definition of the Perron-Bremermann envelope goes back to  \cite{Perron} and \cite{Bremermann}. That this envelope is continuous when $U$ is strictly pseudoconvex was proved by Walsh \cite{Walsh} who also gives examples in which this fails for more general $U$.  That the envelope is locally $\mathcal{C}^{1,1}$ when the domain is the unit ball was proved by Bedford-Taylor \cite{BedfordTaylor1}, where they also showed that for any smoothly bounded strictly pseudoconvex domain the solution was Lipschitz up to the boundary. The full statement of Theorem \ref{thm:locreg} (namely that the solution is $\mathcal{C}^{1,1}$ all the way up to the boundary) is due to Krylov \cite{Krylov} (see also \cite{Guedjbook} for a detailed exposition of Krylov's proof).

Example \ref{ex:gamelin} is due to Gamelin and Sibony (see \cite{Gamelin} and also \cite[Ex. 2.13]{Guedjbook}).   The study of Monge-Amp\`ere foliations goes back to the work of Bedford-Kalka \cite{BedfordKalka}. 

Pluricomplex Green functions were introduced by Klimek \cite{Klimek2} and independently by Zakharyuta \cite{Zakharyuta}. The part of Theorem \ref{thm:pluricomplex} which says that the pluricomplex Green functions solves the HMAE was first proved by Demailly \cite{Demailly87}, while the $\mathcal{C}^{1,1}$-regularity is due to Blocki \cite{Blocki1}.  More on the pluricomplex Green function and its applications can be found, for instance, in \cite{BlockiBergman, Guan1, Guan2,Herbort2, Herbort1,Lempert4}.

\section{The HMAE for compact K\"ahler manifolds}

\subsection{Weak and Regular Solutions}

Suppose now that $(X,\omega)$ is a compact K\"ahler manifold (without boundary) and $\Sigma$ is a compact Riemann-surface with non-empty smooth boundary.   Let $\phi\in\mathcal{C}^{\infty}(X\times \partial \Sigma)$ be chosen so $\phi(\cdot,\tau)\in \mathcal K(X,\omega)$ for each $\tau \in \partial \Sigma$.  Letting $\pi_X:X\times \Sigma\to X$ be the projection, we denote by $\Psh(X\times \Sigma, \pi_X^*\omega)$  the space of functions that are $\pi_X^*\omega$-plurisubharmonic on the interior of $X\times \Sigma$.    

\begin{definition}
We say $\Phi\in \Psh(X\times \Sigma, \pi_X^*\omega)\cap \mathcal{C}^0(X\times \Sigma)$ solves the HMAE with boundary data $\phi$ if
\begin{align}
     MA_{\pi^*\omega}(\Phi)&= 0,\label{eq:hmae}\\
%    \pi_X^*\omega + dd^c \Phi &\ge 0 \nonumber\\
    \lim_{z\to \zeta}\Phi(z,\zeta) &= \phi(\zeta) \text{ for all } \zeta \in X\times \partial \Sigma.\nonumber
\end{align}
\end{definition}
By analogy with Perron-Bremermann  envelope set
$$\Phi:={\sup}^*\{\Psi\in \Psh(X\times \Sigma,\pi_X^*\omega): \limsup_{z\to\zeta}\Psi(z)\leq \phi(\zeta) \text{ for  } \zeta\in X\times \partial \Sigma\}$$

We assume henceforth that $\Sigma$ is compact and carries a smooth strictly plurisubharmonic function $\chi$ such that $\chi|_{\partial \Sigma}=0$ (in fact we will mostly be concerned when $\Sigma$ is either a disc or an annulus in which case this clearly holds).    Using this, the following is proved almost exactly as in the local case  (see, for instance \cite[Ch.\ 1,7]{Guedjbook}).

\begin{proposition} \label{prop:env3}
The envelope $\Phi$ is the unique solution to \eqref{eq:hmae}.
\end{proposition}

We are thus justified in calling $\Phi$ the \emph{weak solution} to the HMAE with boundary data $\phi$.   The following statement (proved recently by Chu-Tosatti-Weinkove \cite{ChuTossatiWeinkove}) is the optimal regularity that one can expect in general.

\begin{theorem}\label{thm:C11}
Let $\phi\in\mathcal{C}^{\infty}(X\times \partial \Sigma)$ be chosen so $\phi(\cdot,\tau)\in \mathcal K(X,\omega)$ for each $\tau \in \partial \Sigma$. Then the weak solution $\Phi$ to the HMAE with boundary data  $\phi$ lies in $\mathcal{C}^{1,1}(X\times\Sigma)$.
\end{theorem}

Observe that by hypothesis $\omega_\tau: = \omega + dd^c\phi(\cdot, \tau)$ is a K\"ahler form on $X$ for each $\tau\in \partial \Sigma$, and so if $\Phi$ solves \eqref{eq:hmae} then the $(1,1)$-current 
$$ \Omega: = \pi_X^* \omega + dd^c\Phi$$
solves the Dirichlet problem for the HMAE with boundary data $\{\omega_{\tau}\}_{\tau\in \partial \Sigma}$, as considered in the introduction.  Following Donaldson \cite{Donaldson} we make the following definition:
%\S\ref{sec:introduction}.

\begin{definition}\label{def:regularcompact}
 We say the weak solution $\Phi$ to the HMAE \eqref{eq:hmae} is \emph{regular} if it is smooth and  $\Phi(\cdot,\tau) \in \mathcal K(X,\omega)$ for all $\tau\in\Sigma$.
\end{definition}

Just as in the local case, a regular solution defines a foliation of $X\times \Sigma$.   In more detail, consider the associated form $\Omega: = \pi_X^*\omega + dd^c\Phi$.  By being the weak solution to the HMAE we have $\Omega\ge 0$ and $\Omega^{n+1}=0$ on $X\times \Sigma$.  On the other hand, if $\Phi$ is regular then $\Omega|_{X\times \{\tau\}}= \omega +dd^c\Phi(\cdot,\tau)$ is strictly positive for all $\tau \in \Sigma$.  Thus the kernel of $\Omega$ at each point of $X\times \Sigma$ is one-dimensional, and so gives a one-dimensional distribution.  Since $\Omega$ is closed the  distribution is integrable, and so by the Frobenius Integrability Theorem gives foliation of $X\times \Sigma$. The leaves are complex since $\Omega$ is of type $(1,1)$ and by construction $\Phi$ is $\pi_X^*\omega$-harmonic along these leaves.   As $\Omega|_{X\times \{\tau\}}$ is strictly positive, these leaves are necessarily transverse to the fibres over $\Sigma$.

If $\Phi$ is merely a weak solution then there is no reason to think such a foliation will exist.  However it can still happen that there are some transverse curves along which the weak solution is harmonic.

\begin{definition}
 Let $f:\Sigma\to X$ be proper and holomorphic.  We say the graph of $f$ is a proper \emph{harmonic curve} for the weak solution $\Phi$ to the HMAE if $\Phi\circ f$ is $\pi_X^*\omega$-harmonic.  If $\Sigma=\mathbb D$ is the unit disc we refer to this as a proper \emph{harmonic disc}.
\end{definition}

It is in general very hard to determine whether a certain weak solution to the HMAE is regular or not.  A trivial, but still sometimes useful, special case is when $\phi\in \mathcal C^{\infty}(X\times \partial \Sigma)$ is independent of the point in $\partial \Sigma$.  For then  $\Phi(z,\tau):= \phi(z)$ for $(z,\tau)\in X\times \Sigma$ is clearly a regular solution to the HMAE, whose Monge-Amp\`ere foliation is simply the horizontal slices $\{z\}\times \Sigma$ for $z\in X$.  We will see in the next section that this can be used to produce non-trivial examples of regular solutions.

\subsection{Donaldson's Openness Theorem} \label{sect:donopthm}

Suppose now that $\Sigma=\mathbb D$ is the unit disc in $\mathbb C$.  The following theorem says that the existence of regular solutions to the HMAE persists under small perturbations of the boundary data.

\begin{theorem}[Donaldson]\label{thm:donaldsonopenness}
Suppose the weak solution to the HMAE with boundary data $\phi\in \mathcal{C}^{\infty}(X\times \partial\mathbb{D})$ is regular.    Then for any $k\geq2$ and $0<\alpha\leq 1$ there is an $\epsilon>0$ such that if $g\in \mathcal{C}^{\infty}(X\times \partial\mathbb{D})$ has $\mathcal C^{k,\alpha}$-norm less than $\epsilon$ then the weak solution to the HMAE with boundary data $\phi+g$ is also regular.
\end{theorem}

Very roughly, this result is obtained by casting the problem of deforming the harmonic discs of a Monge-Amp\`ere foliation as an elliptic problem.  Actually is it not precisely these discs that are used, but instead discs from an auxiliary construction that we now sketch.

Let $\Theta:=\Theta_1+i\Theta_2$ be a holomorphic $2$-form on a complex manifold $W$, where $\Theta_1$ and $\Theta_2$ are real symplectic forms. A (real) submanifold $V$ of $W$ is said to be an \emph{LS-submanifold} if it is Lagrangian with respect to $\Theta_1$ (i.e. $\Theta_1$ vanishes along $W$) while being symplectic with respect to $\Theta_2$ (i.e. $\Theta_2$ restricts to a symplectic form on $W$).

Semmes \cite{Semmes} and Donaldson \cite{Donaldson} show that given a compact K\"ahler manifold $(X,\omega)$ there exists a holomorphic fibre bundle $\pi:W_X\to X$ with holomorphic $2$-form $\Theta$ such that K\"ahler forms in the same cohomology class as $\omega$ correspond to $LS$ submanifolds in $W_X$.  Roughly speaking, $W$ is constructed as follows.  If $\omega$ has a local potential $u$ on some open set $U$ we identify $W_U$ with the $(1,0)$-part of the complexified cotangent bundle of $U$. If $z_i$ are local holomorphic coordinates any $(1,0)$-form  can be written as $\sum_i\zeta_idz_i$, thus $(\zeta_i,z_i)$ are local holomorphic coordinates on $W_U$ and locally $\Theta:=\sum_id\zeta_i\wedge dz_i$. 

If $V$ is another open set where $\omega$ has the local potential $v$, then over $U\cap V$  the transition function of the fibre bundle $W_X$ is set to be $\partial(v-u)$. Thus there is a global section of $W_X$, locally given by $\partial u$. By a simple calculation, the graph of this section is a LS-submanifold.   Any K\"ahler form cohomologous to $\omega$ comes from a K\"ahler potential $\phi\in \mathcal K(X,\omega)$, whose corresponding LS-submanifold is locally given by the graph of $\partial(u+\psi)$.  Moreover, as Donaldson shows, in \cite{Donaldson}, any closed LS-manifold in $W_X$ in the isotopy class of $\partial u$ arises this way.

Now let $\phi\in \mathcal{C}^{\infty}(X\times \partial\mathbb{D})$ and assume that $\phi_{\tau}\in \mathcal{K}(X,\omega)$ for each $\tau\in \partial\mathbb{D}$. By the above, this defines a family $\Lambda_{\tau}$ of associated LS-submanifolds in $W_X$. Donaldson proves the following:

\begin{proposition}
There is a regular solution $\Phi$ to the HMAE with boundary data $\phi$ if and only if there is a smooth family of holomorphic discs $g_x:\mathbb{D}\to W_X$ parametrized by $x\in X$ such that
\begin{itemize}
    \item $\pi(g_x(0))=x$;
    \item for each $\tau\in \partial \mathbb{D}$ and each $x\in X$, $$g_x(\tau)\in \Lambda_{\tau};$$
    \item for each $\tau \in \mathbb{D}$, the map $x\mapsto g_x(\tau)$ is a diffeomorphism of $X$.
\end{itemize}
For a fixed $\tau\in \mathbb{D}$ the image of the map $x\mapsto g_x(\tau)$ is the $LS$-submanifold associated to the K\"ahler form $\omega+dd^c\Phi(\cdot,\tau)$.
\end{proposition}

Thus regular solutions to the HMAE come from these particular families of holomorphic discs.  Then one can apply the deformation theory of holomorphic discs with boundary in a totally real submanifold (which is essentially an elliptic problem) to see that the existence of such a family is open as the boundary data varies, thus proving Theorem \ref{thm:donaldsonopenness}.
 
It is interesting to note that the regularity result of Lempert for the pluricomplex Green function discussed in \S \ref{sect:pluricomplex} is proved in a somewhat analogous manner. Recall that a holomorphic disc $f:\mathbb{D}\to U$ with $f(x)=z$ and $f(y)=w$ is said to be extremal if it realizes the Kobayashi distance between $z$ and $w$. Let $v$ denote the normal vector field of $\partial U$ pointing outward. Lempert calls a disc $f$ stationary if it extends continuously to a map $f:\overline{\mathbb{D}}\to \overline{U}$ with $f(\partial \mathbb{D})\subseteq \partial U$, and if the map $\partial \mathbb{D}\ni \zeta \mapsto [\overline{v_1(f(\zeta))}:...:\overline{v_n(f(\zeta))}]\in \mathbb{P}^{n-1}$ extends to a holomorphic function $\hat{f}:\mathbb{D}\to \mathbb{P}^{n-1}$. Lempert proves that a stationary disc is extremal and conjectures that the converse also holds. One can interpret $f$ being stationary as saying that the combined disc $(f,\hat{f})$ is attached to a certain totally real submanifold, and hence stationary discs persist given small perturbations of   $U$.  In particular this proves regularity for the pluricomplex Green function for domains that are small perturbations of the unit ball (thus the analogy with Donaldson's Openness proof).  To prove the result for all strictly convex domains Lempert uses a continuity argument, by establishing the required a priori estimates.

\subsection{Bibliographical Remarks}

 In work of Mabuchi \cite{Mabuchi}, Semmes \cite{Semmes} and Donaldson \cite{Donaldson},  the space  $\mathcal K(X,\omega)$ of K\"ahler metrics cohomologous to $\omega$ is given the structure of an infinite dimensional Riemannian manifold and, somewhat amazingly, the HMAE turns out to be the geodesic equation in this space.   More specifically, to find a geodesic segment joining two points $\phi_0,\phi_1\in \mathcal K(X,\omega)$ requires solving the Dirichlet problem for the HMAE over $X\times A$ where $A$ is an annulus, say $A = \{ c_0<|\tau|<c_1\}$, and the boundary data is taken to be $\phi(z,\tau) := \phi_i(z)$ for $|\tau|=c_i$ with $i=0,1$.    Thus any smoothness properties of the weak solution to the HMAE becomes a statement about smoothness of this (weak) geodesic segment, and having a regular solution says precisely that there is a genuine (i.e.\ smooth) geodesic segment joining $\phi_0$ and $\phi_1$ in $\mathcal K(X,\omega)$.      This manifestation of the HMAE generated much interest, not least since it was observed by Donaldson \cite{Donaldson} that the existence of a (sufficiently nice) geodesic segment joining any two points in $\mathcal K(X,\omega)$ would imply uniqueness of constant scalar curvature K\"ahler metrics.

We refer the reader again to \cite{Berndtsson2,Guansurvey,Guedjbook,Kol1, Kol2,PhongSongSturmSurvey,Zeriahisurvey} for other surveys on this topic.  The statement that the weak solution to the HMAE is $\mathcal{C}^{1,1}$, now proved in \cite{ChuTossatiWeinkove}, has a long history.    It was proved by Chen \cite{ChenI} (with complements by B\l ocki \cite{Blocki2}) that the weak solution has bounded Laplacian on $X\times \Sigma$ and so in particular is $\mathcal{C}^{1,\alpha}$ for any $\alpha<1$ in the interior of $X\times \Sigma$.  Moreover B\l ocki proves that if $(X,\omega)$ is assumed to have non-negative bisectional curvature then the weak solution is $\mathcal{C}^{1,1}$.   Other works on this topic include those of Phong-Sturm \cite{PhongSturmMAgeodesic,PhongSturmDirichlet,PhongSturmRegularity}, Eyssidieux--Guedj--Zeriahi \cite{Eyssidieux}, Demailly et al \cite{Demailyetal}.  When $\Sigma$ is the unit disc, $\mathcal{C}^{1,1}$ on the interior of $X\times \Sigma$ has been proved by Berman \cite{Berman3} using a technique based on the original approach of Bedford-Taylor.   

One can more generally consider the Dirichlet problem for the HMAE on a complex manifold-with-boundary, and several of the above cited references, including \cite{ChuTossatiWeinkove},  hold in this case as well (usually under an assumption of being weakly-pseudoconcave or having Levi-flat boundary).  For example, one can consider the HMAE on the total space of a (sufficiently nice) test-configuration, thus connecting K-stability with weak-geodesics (see, for instance \cite{BermanKpoly,ChenTang,PhongSturmTestConfigKstab,RossNystromATC, SongZelditch} as well as the contribution by Sz\'ekelyhidi in this volume).   

Works on the related question of the implications of the HMAE to the geometry of the space of K\"ahler metrics include those of   Arezzo-Tian \cite{ArezzoTian}, Berman-Boucksom-Guedj-Zeriahi \cite{Bermanetal}, Berndtsson-Cordero-Erausquin-Klartag-Rubinstein \cite{BerndtssonYanir},  Chen-Sun \cite{ChenSun}, Chen-Tian \cite{ChenTian} and Darvas \cite{DarvasI, DarvasII}.    Ultimately it turned out that the particular application concerning uniqueness of constant scalar curvature K\"ahler metrics cannot easily be addressed through regularity, but can resolved with just the weak solution as achieved by Berman-Berndtsson \cite{BerndtssonBerman} (see also Chen-Li-P\u aun \cite{ChenPaun}).

  Donaldson \cite{Donaldson} gives examples of boundary data over the disc for which the weak solution is not regular, but we observe that the argument uses contradiction, and thus is non-explicit.  Nevertheless, it was initially hoped that this phenomena would not hold over the annulus, and so  any weak geodesic connecting two K\"ahler potentials would be regular (and thus a geodesic in the strongest possible sense).   It was not until the work of Lempert-Vivas that this was proven not to be the case.  In \cite{LempertVivas} they find geodesic segments that are not $\mathcal{C}^3$ up to the boundary, and later Darvas-Lempert \cite{LempertDarvas} found geodesic segments that fail to be $\mathcal{C}^2$ up to the boundary.  In subsequent sections we will see how regularity can fail both for the HMAE over the disc and over the punctured disc.  As the case for the pluricomplex Green's function, it is currently unknown whether or not singularities can occur in the interior.

%\subsection{Harmonic discs in the space of K\"ahler metrics}

%Let $(X,\omega)$ be a compact K\"ahler manifold and let $\Sigma$ stand for either the unit disc or an annulus in $\mathbb{C}$.
%Let $\phi$ be a smooth function on $X\times \partial \Sigma$ and assume that for all $\tau\in \partial \Sigma$ the restriction of $\phi$ to the fiber over $\tau$ is a K\"ahler potential with respect to $\omega$. The associated Dirichlet problem for the HCMAE asks for an $\pi^*\omega$-psh function $\Phi$ on $X\times \Sigma$ such that $MA_{\pi^*\omega}(\Phi)=0$ and $\lim_{z\to\zeta}\Phi(z)=\phi(\zeta)$ for all $\zeta\in X\times \partial \Sigma$.

%\subsection{Donaldson Openness Theorem}

%Recall from Section \cite{Sect:geod} that a solution $\Phi$ is said to be regular if it is smooth and that the restriction of $\Phi$ to each $X$-fiber is a K\"ahler potential. In other words the kernel of $\pi^*\omega+dd^c\Phi$ should be transversal to the $X$-fibers, which in particular means that it must be one dimensional. We then get a Monge-Amp\`ere foliation of $X\times \Sigma$, and each of its leaves will be naturally parametrized by $\Sigma$.

\section{The Hele-Shaw Flow}\label{sec:HS}

The rest of this paper is devoted to surveying previous work of the authors which connects the HMAE with the Hele-Shaw flow.  We shall discuss two approaches to this flow, and both are useful in understanding its relation with the HMAE.  First is the so-called \emph{weak Hele-Shaw flow} that can be described using basic potential-theoretic constructions.  The advantage of this approach is that it does not require any a priori smoothness, making it both elementary and very flexible.  Second is the \emph{strong Hele-Shaw flow} that is defined dynamically by describing the motion of the boundary of the flow.  This necessarily requires assuming more smoothness, but has the advantage of having a physical interpretation thus making it more intuitive.

Of course, the strong Hele-Shaw flow is also a weak one, and a weak Hele-Shaw flow that is also smooth will be a strong one.  In this section we shall only consider the weak flow, allowing us to quickly move to the connection with the HMAE.    Consideration of the strong flow will be postponed until \S\ref{sec:strong}.

Our account is broadly self-contained, in that we include all the main features of the flow we need.  That said, this represents only a tiny part of the Hele-Shaw flow theory, and the reader will find in the bibliographical remarks many references that go far beyond what is included here.

\subsection{Lelong Numbers}\label{sec:lelong}

From now on, $X$ will be a connected Riemann surface along with a distinguished point $z_0\in X$ and $\omega$ will be a K\"ahler form on $X$.   As $X$ has complex dimension $1$ being plurisubharmonic is the same as being subharmonic, and we let $\Sh(X,\omega)$ denote the space of functions that are $\omega$-subharmonic.  Let $z$ be a holomorphic coordinate defined near $z_0$.  Then for $\psi\in \Sh(X,\omega)$ the \emph{Lelong number} of $\psi$ of $z_0$ is defined to be
$$\nu_{z_0}(\psi) := \sup \{c\ge 0: \psi \le c \ln |z-z_0|^2 + O(1)\}$$
where the inequality is to be understood as meaning there is a constant $C$ such that $\psi\le c\ln |z-z_0|^2 +C$ near $z_0$.  We observe the supremum is actually attained, so if $\nu_{z_0}(\psi)=t$ then $\psi\le t \ln |z-z_0|^2 + O(1)$.  To see this,  let $B$ be a small ball centered around $z_0$.  For any $c<t$ the function $\psi(z) - c\ln |z-z_0|^2$ is bounded above as $z$ tends to $z_0$, and lies in $\Sh(B\setminus \{z_0\},\omega)$ and thus extends to a function in $\Sh(B,\omega)$ \cite[Theorem 2.7.1]{Klimek}.    On the other hand, on the boundary of the ball,  $\psi(z) - c\ln |z-z_0|^2|_{\partial B}$ is bounded from above uniformly over all $c<t$.  Thus by the maximum principle $\psi(z) - c\ln |z-z_0|^2$ is bounded above uniformly over $z\in B$ and $c<t$.  Then letting $c$ tend to $t$ gives $\psi(z) \le t\ln |z-z_0|^2 + O(1)$ as claimed.

The Lelong number measures the mass of the current $dd^c\psi$ at the point $z_0$, in that
\begin{equation}\label{eq:lelongnumbermass}
\nu_{z_0}(\psi) = \lim_{r\to 0^+} \int_{B_r} dd^c\psi = \lim_{r\to 0^+} \int_{B_r} \omega_\psi
\end{equation}
where $B_r$ is the ball of radius $r$ centered at $z_0$ \cite[Theorem 2.8]{Demaillytrento}.

\subsection{Definitions}\label{sec:hsdefinitions}
%\begin{}
%For our applications we have the following three examples in mind:
%\begin{enumerate}
%\item $\Sigma$ is an open subset of $\mathbb C$ and $g$ is the metric whose K\"ahler form is given by
%$$ \omega = \frac{dz\wedge d\overline{z}}{(1+|z|^2)^2}$$
%\item $\Sigma$ 
%\end{enumerate}
%\end{example}
 
The basic definition on which everything else is based is the following:
\begin{definition}(Hele-Shaw Envelope)
For $t\in \mathbb R$ let 
$$\psi_{t} := \sup\{\psi\in \Sh(X,\omega)  : \psi\le 0 \text{ and } \nu_{z_0}(\psi) \ge t\}.$$
We shall refer to $\psi_t$ as the \emph{Hele-Shaw envelope} at time $t$.
\end{definition}

Of course the envelope $\psi_t$ depends on the background K\"ahler form $\omega$, but this will always be clear from context.  Clearly $\psi_t\le 0$ everywhere.

\begin{definition}(Weak Hele-Shaw Flow)
For $t\in \mathbb R$ set
$$\Omega_t :  = \{z\in X : \psi_t(z)< 0\}.$$
We refer to $\Omega_t$ as the \emph{weak Hele-Shaw domain} at time $t$, and the collection of all such domains as the \emph{weak Hele-Shaw flow}.
\end{definition}

The weak Hele-Shaw domains are generally hard to compute, unless one imposes some additional symmetry as in the following example.

\begin{example}(Radially Symmetric Case)\label{ex:radiallysymmetric}
Suppose $X=\mathbb C$, let $z_0$ be the origin and assume the K\"ahler form $\omega$ is radially symmetric.  Then we can write
$$ \omega = dd^c \phi$$
for some smooth radially symmetric function $\phi$ on $\mathbb C$, so
$$ \phi( e^{i\theta} z ) = \phi(z) \text{ for all }  \theta\in \mathbb R.$$
It is not hard to see that the Hele-Shaw envelopes and Hele-Shaw domains are also be radially symmetric, and we now calculate what these actually are.  We assume for all $t>0$ that $\phi$ satisfies the growth condition
$$\phi(z) \ge t \ln |z|^2 + O(1) \text{ for } |z|\gg 0.$$
It is convenient to use the variable
$$ s= -\log |z|^2$$
so our distinguished point $z=0$ corresponds to $s=\infty$.   Then we can write
$$ \phi(z) = u(s)$$
for some smooth $u:\mathbb R\to \mathbb R$.  By differentiating twice, one can check the condition that $\omega$ is strictly positive implies $u$ is strictly convex,  and
$$ \lim_{s\to \infty} \frac{du}{ds} =0 \text{ and }  \lim_{s\to -\infty} \frac{du}{ds} =\infty$$
(the first coming from $\phi(z)$ extending smoothly over $z=0$,  and the second coming from the assumed growth condition).     So for $t\in \mathbb R_+$  there is a unique $s_0\in \mathbb R$ such that
$$\frac{du}{ds}|_{s_0}=-t.$$
We let
$$ v_t(s) := \left\{ \begin{array}{ll} u(t) & \text{ for } s<s_0 \\ u(s_0) -t(s-s_0) &\text{ for }s\ge s_0\end{array}\right.$$
Then $v_t$ is  the largest convex function bounded above by $u$ with the property that $v_t(s) \le -ts + O(1)$ as $s\to \infty$.     We claim the Hele-Shaw envelope is given by
\begin{equation}
\psi_t(z) = v_t(s)-u(s)\label{eq:radialclaim}
\end{equation}
and the weak Hele-Shaw domain is
$$ \Omega_t = \{ s>s_0 \} = \{ z : |z|^2<e^{-s_0}\}.$$

To prove this,  set $\tilde{\psi}(z) = v_t(s)-u(s)$ so the goal is to show $\tilde{\psi} = \psi_t$.    Observe $v_t$ being convex implies $\tilde{\psi}\in \Sh(\mathbb C,\omega)$ and its behaviour as $s$ tends to infinity gives $\nu_{z=0}({\tilde{\psi}})=t$.  Clearly $\tilde{\psi}\le 0$, so $\tilde{\psi}\le \psi_t$.    For the other inequality, let $\psi\in \Sh(X,\omega)$ satisfy $\psi\le 0$ and $\nu_{z=0}(\psi)\ge t$.   As $v_t$ is linear on $\{s>s_0\}$ we have $\omega_{\tilde{\psi}}=0$ on $D^\times: = \{s>s_0\} = \{ 0<|z|^2<e^{-s_0}\}$.    Then the difference $\psi-\tilde{\psi}$ is bounded as $z\to 0$ and subharmonic on $D^\times$ and thus extends to a subharmonic function on all of $D$ \cite[Theorem 2.7.1]{Klimek}.  On the other hand $\tilde{\psi} = 0$ on $\partial D$, and so $\psi-\tilde{\psi}\le 0$ on $\partial D$.  Thus by the maximum principle, $\psi\le \tilde{\psi}$ on all of $D$.  But $v_t=u$ on the set $\{s\ge s_0\}$ so on the complement of $D$ clearly $\tilde{\psi} = 0\ge \psi$, and hence $\psi\le \tilde{\psi}$ everywhere.   Taking the supremum over all such $\psi$ gives $\psi_t\le \tilde{\psi}$, and thus $\psi_t=\tilde{\psi}$ as claimed.  The conclusion then about the weak Hele-Shaw domain follows as this is the set on which $v_t$ is equal to $u$.\end{example}

\subsection{Basic Properties of the Hele-Shaw Flow on compact Riemann surfaces}

Assume now $X$ is compact, which in particular implies $\int_X \omega$ is finite.  It is not hard to see if $\omega$ is replaced with $\lambda\omega$ for some $\lambda>0$ then $\psi_t$ is replaced with $\lambda \psi_{\lambda^{-1} t}$ and $\Omega_t$ replaced by $\Omega_{\lambda^{-1} t}$.  Thus without loss of generality we assume that
$$ \int_X \omega =1.$$
With this in mind we turn to some of the basic properties of the weak Hele-Shaw flow.

\begin{proposition}[Basic Properties of the weak Hele-Shaw flow in the compact case]\label{prop:HSlelong}\ 
\begin{enumerate}
\item For $t\le 0$ we have $\psi_t\equiv 0$ and $\Omega_t=\emptyset$.  
\item For $t>1$ we have $\psi_t\equiv -\infty$ and $\Omega_t =X$. 
\item For $t\in [0,1]$ we have
\begin{enumerate}
\item $\psi_t$ is locally bounded away from $z_0$.
\item $\psi_t\in \Sh(X,\omega)$.
\item $\nu_{z_0}(\psi_t)= t$.
\item $\omega_{\psi_t}|_{\Omega_t} =t\delta_{z_0}$.
\end{enumerate}
\end{enumerate}
\end{proposition}
Our proof will use the following preliminary result.

\begin{lemma}\label{lem:existenceofalphat}
There exists an $\alpha\in \Sh(X,\omega)\cap \mathcal{C}^{\infty}(X\setminus\{z_0\})$ such (1) $\sup_X \alpha= 0$ (2) $\alpha = \ln |z-z_0|^2 + O(1)$ near $z_0$, so in particular $\nu_{z_0}(\alpha) = 1$ and (3) $\omega + dd^c\alpha = \delta_{z_0}$.
\end{lemma}
\begin{proof}
Suppose  $z$ is a holomorphic coordinate on a ball $B$ around $z_0$.  Let $\rho$ be a bump function identically 1 near $z_0$ and supported in $B$ and consider
$$ \beta(z):= \rho(z) \log |z-z_0|^2.$$
Then $dd^c\beta = \delta_{z_0} + \tau$ for some smooth form $\tau$.  But in Dolbeault cohomology $0=[dd^c\beta] = [\delta_{z_0}] + [\tau] = [\omega] + [\tau]$ where the last equality uses $\int_X \omega=1$ (and we are using Dolbeault cohomology of currents, which agrees with Dolbeault cohomology of smooth forms \cite[IV, 6.13]{Demailly2}).   Thus $\tau = - \omega + dd^c f$ for some smooth function $f$ on $X$, and $\alpha := \beta -f-C$ for a suitable constant $C$ is in $\Sh(X,\omega)\cap \mathcal{C}^{\infty}(X\setminus\{z_0\})$ and satisfies conditions (1) through (3).
\end{proof}

\begin{remark}
On $\mathbb P^1$, with its Fubini-Study form, and coordinate $z$ on $\mathbb C\subset \mathbb P^1$ so $z_0$ is the origin,  we can explicitly write $\alpha =\ln |z|^2 - \ln (1+|z|^2)$.   
\end{remark}

\begin{proof}[Proof of Proposition \ref{prop:HSlelong}]
All of this is rather standard, and for convenience we give details.  If $t\le 0$ then the constant function $0$ is a candidate for the envelope defining $\psi_t$, giving (1).  On the other hand if $\psi\in \Sh(X,\omega)$ is not identically $-\infty$ and $\nu_{z_0}(\psi)\ge t$ then $t \le \int_X \omega_{\psi} = \int_X \omega=1$ by \eqref{eq:lelongnumbermass} which proves (2).

So assume now $t\in [0,1]$.  Then (3a) follows as $\psi_t$ is bounded from below by the function $t\alpha$ where $\alpha$ is provided by Lemma \ref{lem:existenceofalphat}.  Moreover this implies $\nu_{z_0}(\psi_t)\le \nu_{z_0}(t\alpha) = t$. Now let $z$ be a holomorphic coordinate defined near $z_0$ and consider
$$\beta := \operatorname{sup}^* \{ \psi\in \Sh(X,\omega): \psi\le 0 \text{ and } \psi \le t \ln |z-z_0|^2 + O(1)\}.$$
Clearly $\beta\ge \psi_t$ and we shall show that in fact equality holds.

First observe being the upper-semicontinuous regularisation of a supremum of $\omega$-sub\-harmonic functions,  $\beta$ is itself $\omega$-subharmonic \cite[Thm 2.6.1(iv)]{Klimek} and clearly $\beta\le 0$.     We claim $\nu_{z_0}(\beta)\ge t$.   The issue here is that the $O(1)$ term in the definition of $\beta$ can depend on $\psi$.    To address this, let $B$ be  a small ball around $z_0$ on which we can write  $\omega = dd^c\zeta$ for some smooth function $\zeta$.  Let $\gamma$ be the solution to the classical Dirichlet problem for the Laplacian
\begin{equation}
 dd^c \gamma =0 \text{ on } B \text{ and } \gamma|_{\partial B} = (\zeta-t\ln |z-z_0|^2)|_{\partial B}.\label{eq:dirichletlaplacian}
\end{equation}
It is known \cite[Theorem 2.2.6]{Klimek} such a $\gamma$ exists, and is locally bounded on $B$.  Then set
$$ \epsilon := - \zeta + t \ln |z-z_0|^2$$
and we claim $\beta \le \epsilon$ near $z_0$.    To see this, suppose $\psi\in \Sh(X,\omega)$ is such that $\psi\le 0$ and $\nu_{z_0}(\psi)\ge t$.   Then
$$\psi-\epsilon = \psi + \zeta - t\ln |z-z_0|^2\in \Sh(B\setminus \{z_0\}).$$
%$$ dd^c(\psi - \epsilon) = \omega + dd^c \psi - dd^c\zeta - dd^c \epsilon = \omega_{\psi} \ge 0 \text{ on } B\setminus \{z_0\}.$$
 On the other hand by construction $(\psi-\epsilon)|_{\partial B} \le -\epsilon|_{\partial B} =\gamma|_{\partial B}$.  As $z$ approaches $0$ we have $\psi\le t \ln |z-z_0|^2 + O(1)$ and $\epsilon = t\ln |z-z_0|^2 + O(1)$ so $\psi - \epsilon$ is bounded near $z_0$, and thus extends to a subharmonic function over all of $B$  \cite[Theorem 2.7.1]{Klimek}.  Hence by the maximum principle $\psi\le \epsilon+\gamma $ over $B$.   Taking the supremum over all such $\psi$, and then the upper semicontinuous regularisation, we deduce $\beta\le \epsilon$ near $z_0$ as claimed.  In particular $\beta \le t \ln |z-z_0|^2 + O(1)$ giving $\nu_{z_0}(\beta)\ge t$.   Thus $\beta$ is a candidate for the envelope defining $\psi_t$, so in fact $\beta = \psi_t$ proving items (3b) and (3c).
 
That $\psi_t$ is $\omega$-harmonic away from $z_0$ is proved the same way that the Perron-Bremermann envelope is shown to solve the HMAE.  Then (3d) follows from (3c) and \eqref{eq:lelongnumbermass}.

%Statement (3d) is proved in the same way  let $z\in \Omega_t\setminus \{z_0\}$.  So $\psi_t(z)<-\delta$ for some $\delta>0$.  Let $B$ be a small ball centered at $z$ on which we write $\omega|_B = dd^c\zeta$ and let $\beta$ be the solution to the Dirichlet problem $dd^c \beta =0$ on $B$ and $\beta|_{\partial B} = \psi_t+\zeta|_{\partial B}$.    Then as $\psi_t+ \zeta$ is subharmonic on $B$, by the maximum principle $\psi_t + \zeta|_B \le \beta$.  On the other hand, let $\gamma$ be equal to $\beta-\zeta$ on $B$ and equal to $\psi_t$ on $X\setminus B$.  One can verify that as long as $B$ is sufficiently small, $\gamma\in \Sh(X,\omega)$ is bounded above by $0$ and $\nu_{z_0}(\gamma)\ge t$.  Hence $\gamma\le \psi_t$ so in particular $\beta \le \psi_t+ \zeta|_B$.      Thus  $\psi_t + \zeta = \beta$ on $B$ giving $\omega_{\psi_t} = dd^c \beta =0$ on $B$.  Thus $\omega_{\psi_t} =0$ on $\Omega_t\setminus \{z_0\}$ and so (3d) follows from (3c) and .
\end{proof}

%\begin{remark} Alternatively one can prove $\nu_0(\psi_t)\ge t$ using semicontinuity properties of Lelong-numbers
%\cite[Ch.\ III Prop.\ 5.2]{Demailly2}. 
%%Note to self: write $\psi_t= \lim_{n\to \infty} \psi_{t,n}$ where
%%\psi_{t,n} = \sup \{ \psi\in \Sh(X,\omega) : \psi\le \phi \text{ and } \psi\le t\ln |z|^2 + n \text{ on } B_{1/n}\}$$
%%where $B_{1/n}(0)$ is the ball of radius $1/n$ around $0$.  Then $\psi_{t,n}$ is an increasing sequence in $n$ and so $$\omega_{\psi_{t,n}}\to \omega_{\psi_t}$ as $n$ tends to infinity.  We then apply Demaillybook Ch.\ III Prop 5.2 with the potential $\phi = \ln |z|^2$ (see page 159 of Demaillybook) to get that $\nu(dd^c \psi_{t,n},\phi) = \nu_0(\phi_{t,n})$ so the Proposition gives what we want.
%\end{remark}

\begin{corollary}(Openness, Connectedness) \label{cor:open}\ 
The Hele-Shaw domain $\Omega_t$ is open, connected and $z_0\in \Omega_t$ for $t>0$. 
\end{corollary}
\begin{proof}
Openness of $\Omega_t$ follows from semicontinuity of $\psi_t$, and if $t>0$ then $\nu_{z_0}(\psi)>0$ so  $\psi_t(z_0)=-\infty$ giving $z_0\in \Omega_t$.   If $\Omega_t$ were not connected then we could find a component $S$ that does not contain $z_0$.  Since $\Omega_t$ is open,  $\partial S\subset X\setminus \Omega_t$ and so $\psi_t=0$ on $\partial S$.    As $\omega_{\psi_t}=0$ on $\Omega_t\setminus \{z_0\}$ we see $-\psi_t$ is subharmonic on $S$, so the maximum principle implies $-\psi_t\le 0$ on $S$.  But this is absurd as $S\subset \Omega_t = \{ \psi_t<0\}$.
\end{proof}

The next two results show the Hele-Shaw domains only depends on the value of the K\"ahler metric in a region slightly larger than that domain.  To express this precisely, suppose $\tilde{\omega}$ is another K\"ahler form on $X$, with $\int_X \tilde{\omega}=1$, and denote by $\tilde{\psi}_t$ and $\tilde{\Omega}_t$ the Hele-Shaw envelopes and weak Hele-Shaw domains associated to $\tilde{\omega}$.

\begin{lemma}[Monotonicity]\label{lem:monotonicity}\ 
Suppose $S\subset X$ is open and 
$$\tilde{\omega} \ge \omega \text{ over S},$$
and assume $\Omega_t$ is relatively compact in $S$.  Then 
\begin{equation}\psi_t \le \tilde{\psi}_t \text{ and }
\tilde{\Omega}_t \subset \Omega_t.\label{eq:monosets}
\end{equation}
\end{lemma}
\begin{proof}
The statement is trivial if $t<0$ or $t>1$, so suppose $t\in [0,1]$.  From Proposition \ref{prop:HSlelong}(3b)  $\psi_t\in \Sh(X,\omega)$, so the hypothesis implies $\tilde{\psi}_t\in \Sh(S,\tilde{\omega})$.  Since $\Omega_t$ is relatively compact in $S$ we see $\psi_t$ is identically zero on a neighbourhood of $X\setminus S$, and so over this neighbourhood $\tilde{\omega}_{\psi_t} = \tilde{\omega}\ge 0$.  Thus $\psi_t\in \Sh(X,\tilde{\omega})$.  Now Proposition \ref{prop:HSlelong}(3c) gives $\nu_{z_0}(\psi_t)\ge t$, and so $\psi_t$ is a candidate for the envelope defining $\tilde{\psi}_t$, giving $\psi_t\le \tilde{\psi}_t$ from which it follows $\tilde{\Omega}_t\subset \Omega_t$
\end{proof}

\begin{corollary}(Locality)\label{cor:locality}
If $\omega= \tilde{\omega}$ on some open $S\subset X$ and $\Omega_t$ is relatively compact in $S$ then $\psi_t = \tilde{\psi}_t$ and $\Omega_t = \tilde{\Omega}_t$.
\end{corollary}
\begin{proof}
One application of the previous lemma tells us $\tilde{\Omega}_t\subset \Omega_t$ and so $\tilde{\Omega}_t$ is also relatively compact in $S$.  Then we can apply the lemma again with the roles of $\tilde{\omega}$ and $\omega$ reversed.
\end{proof}

We next express the Hele-Shaw envelope in a slightly different way.     Recall the function $\alpha$ from Lemma \ref{lem:existenceofalphat} that is smooth away from $z_0$, and satisfies 
\begin{align*}
\omega+dd^c\alpha &=\delta_0 \text{ and } \sup_X \alpha = 0 \text{ and }
\alpha = \ln |z-z_0|^2 + O(1) \text{ near }z_0.
\end{align*}

\begin{lemma}\label{lem:alternativeenvelope}
For $t\in [0,1]$,
$$\psi_t = \sup\{\psi\in \Sh(X,(1-t)\omega) : \psi\le -t\alpha\} + t\alpha.$$
%and
%\begin{equation}\psi_1 = \alpha.\label{eq:psi1isalpha}\end{equation}
\end{lemma}
\begin{proof}
The statement is trivial when $t=0$, so we assume $t>0$.  Set 
$$u:=\sup\{\psi\in \Sh(X,(1-t)\omega) : \psi\le -t\alpha\}$$ so the goal is to prove $\psi_t = u + t\alpha$.  Clearly $\psi_t-t\alpha\le -t\alpha$ and $(1-t)\omega + dd^c(\psi_t-t\alpha) = \omega_{\psi_t} - t\delta_0\ge 0$ by Proposition \ref{prop:HSlelong}(3b,d).   On the other hand since $\nu_{z_0}(\psi_t)= t$ we have $\psi_t-t\alpha$ is bounded near $z_0$.  Thus $\psi_t-t\alpha$ extends over $z_0$ to a function in $\Sh(X,(1-t)\omega)$ \cite[Theorem 2.7.1]{Klimek} and we conclude $\psi_t-t\alpha\le u$.  

For the other inequality, if $\psi\in \Sh(X,(1-t)\omega)$ satisfies $\psi\le -t\alpha$ then $\psi+t\alpha\le 0$ and $t\alpha\in \Sh(X,t\omega)$ so by convexity $\psi+t\alpha\in \Sh(X,\omega)$.  Moreover any such $\psi$ is bounded above near $z_0$,  so $\nu_{z_0}(\psi + t\alpha) \ge \nu_{z_0} (t\alpha) = t$.  Hence $\psi + t\alpha\le \psi_t$, and taking the supremum over all such $\psi$ gives $u+ t\alpha\le \psi_t$ as required.

%For the second statement, if $\psi\in \Sh(X)$ with $\psi\le 0$ and $\nu_{z_0}(\psi)=1$ then $\psi-\alpha$ extends to a function in $\Sh(X)$.  By the maximum principle, and the assumption that $\sup_X \alpha=0$ this implies $\psi-\alpha\le 0$, giving $\psi_1\le \alpha$.  But clearly $\alpha\le \psi_1$ and hence they are equal.\todo{JR: note to david: The second statement was added 29 March, and it means I am removing the assumption that $t$ is less than 1 in Theorem 5.11,Corollary 5.12}
\end{proof}

%\begin{example} The monotonicity also provides some information as to how the weak Hele-Shaw behaves.  Suppose $X=\mathbb P^1$ and $z_0$ be the point $z=0$ as in Example \ref{ex:radiallysymmetric} and $\omega$ is the Fubini-Study form. Let $\phi\in \mathcal K(X,\omega)$.  Then it is possible to find a radially symmetric $\tilde{\phi}\in \mathcal K(X,\omega)$ with the property that $\omega_{\tilde \phi} \ge \omega_\phi$ near $z_0$.    But in Example \ref{ex:radiallysymmetric} we computed explicitly that this latter flow consists of discs centered at the origin of a particular radius.  This gives a lower-bound on the weak Hele-Shaw flow $\Omega_t$ taken with respect to $\phi$, since it must be as least as large as these discs.\end{example}

The previous Lemma casts the envelope $\psi_t$ as a (translation of) the solution to an obstacle problem with obstacle $-t\alpha$.  A slight difference between this and the classical theory is that often the obstacle is assumed to be a smooth (or at least bounded) function, but this is easily circumvented as in the following statement.

\begin{lemma}\label{lem:genuineobstacle}
There exists an $f\in \mathcal{C}^{\infty}(X)$ such that
\begin{equation}\sup \{ \psi\in \Sh(X,(1-t)\omega) : \psi\le -t\alpha\} = \sup\{\psi\in \Sh(X,(1-t)\omega) : \psi\le f\}.\label{eq:genuineobstacle}
\end{equation}
\end{lemma}
\renewcommand{\proofname}{Sketch Proof}
\begin{proof}
In a small disc $D$ around $z_0$ on which $\omega = dd^c\zeta$ let $v$ solve $dd^c v=0$ on $D$ and $v|_{\partial D} = -t\alpha + (1-t)\zeta|_{\partial D}$.  If $\psi$ is a candidate for the envelope on the left hand side of \eqref{eq:genuineobstacle} then by the maximum principle $\psi\le v-(1-t)\zeta=:w$ on $D$.   Now $w$ is bounded but $-t\alpha$ tends to infinity near $z_0$, so we can find an $f\in \mathcal{C}^{\infty}(X)$ such that $f=-t\alpha$ on $X\setminus D$ and $f\le -t\alpha$ on $X$ and $w\le f$ on $D$, and it is easy to see then that \eqref{eq:genuineobstacle} holds for this $f$.
\end{proof}
\renewcommand{\proofname}{Proof}

For more advanced information about the flow we will need some smoothness of the Hele-Shaw envelope.  Note  $\psi_t$ will not generally be $\mathcal{C}^{\infty}$, as can be seen in Example \ref{ex:radiallysymmetric}.  However the following says  this is, in some sense, the worst that can happen:

\begin{theorem}[Regularity of Hele-Shaw envelope] \label{lemma:basicHS}  For $t<1$ the Hele-Shaw envelope $\psi_t$ is $\mathcal{C}^{\infty}$ on $\Omega_t\setminus \{z_0\}$ and is $\mathcal{C}^{1,1}$ on $X\setminus \{z_0\}$.
\end{theorem}

\begin{proof}
The first statement is clear as $\omega_{\psi_t} =0$ on $\Omega_t\setminus \{z_0\}$ and harmonic functions are smooth.   The deeper statement is the second, which is somewhat technical and so we omit the details.  When $X=\mathbb P^1$ the result we want may be reduced to known regularity of solutions of the obstacle problem for the Laplacian for domains in $\mathbb R^2$   \cite{CaffarelliK} due to Cafarelli-Kinderlehrer, and the reader interested in this reduction will find details in \cite[Proposition 1.1]{RWHarmonicDiscs}.   For general Riemann surfaces we need more machinery.   For instance, there is no loss in assuming $\omega$ is integral, at which point the $\psi_t$ is among the envelopes considered by \cite{Berman} and \cite{RossNystromEnvelopes} where the desired $\mathcal{C}^{1,1}$ regularity is proved (strictly speaking the cited results only apply when $t$ is rational, but the proofs given there give uniform estimates of the $\mathcal{C}^{1,1}$ under perturbations of $t$ and the result for all $t\in (0,1)$ then follows by approximation).       We refer the reader to \S\ref{sec:bibremarksHS} for further regularity results in this direction.
\end{proof}

\begin{corollary}
For $t<1$ the boundary $\partial \Omega_t$ of the weak Hele-Shaw domain has measure zero.
\end{corollary}
\begin{proof}
Let $u:= -\psi_t$ so $\Omega_ t=\{ u>0\}$ and by the previous Theorem $u$ is $\mathcal{C}^{1,1}$ in a neighbourhood $U$ of $\partial \Omega_t$.  Since $\omega_{\psi_t}|_U=0$, we have $\Delta u\ge \lambda>0$ on $U$.  

Fix $x\in\partial \Omega_t$.   We first claim there is an $\epsilon>0$ such that for all $r>0$ sufficiently small there is a $y$ with
\begin{equation}
u(y) \ge \epsilon r^2 \text{ and } y\in B_r(x).\label{eq:measurezeroclaim}\end{equation}
To prove this, we may work locally near $x$ and assume our distance function is the usual Euclidean one.   Consider a sequence of points $x_n\in \Omega_t$ converging to $x$ as $n$ tends to infinity.  For small $r>0$ consider $n$ sufficiently large so $B_r(x_n)\cap \partial \Omega$ is non-empty.  Set
$$ v(z) : = u(z) - u(x_n) - \epsilon |z-x_n|^2$$
for $\epsilon \ll \lambda$.  Then $v(x_n)=0$ and $\Delta v\ge 0$ on $B_r(x_n)\cap \Omega_t$.   Thus by the maximum principle applied to $v$ on $B_r(x_n)\cap \Omega_t$ we know there is a $y_n\in \partial (B_r(x_n)\cap \Omega_t)$ with $v(y_n)\ge 0$.  Now $\partial (B_r(x_n)\cap \Omega_t) \subset \partial \Omega_t \cup \partial B_r(x_n)$, and if $y_n\in \partial \Omega_t$ then $u(y_n)=0$, so $v(y_n)<0$ which is absurd.  Hence $y_n\in \partial B_r(x_n)$, so in fact $|y_n-x_n|=r$ and $v(y_n)\ge 0$ becomes
$$ u(y_n) \ge u(x_n) + \epsilon r^2.$$
Letting $n$ tend to infinity and taking a subsequence, we deduce there exists a $y\in X$ satisfying \eqref{eq:measurezeroclaim} as claimed.

We next claim there exists a $c\in (0,1)$ such that for any sufficiently small $r>0$ there exists a $y \in B_r(x)$ and
\begin{equation} 
B_{c r} (y) \subset \Omega_t.\label{eq:measurezeroclaimII}
\end{equation}
To see this let $y$ be as in  \eqref{eq:measurezeroclaim}.  The Lipschitz bound on $\nabla u$ near $\partial \Omega_t$, and the fact that $u\equiv 0$ and $\nabla u \equiv 0$ on $\partial \Omega_t$ implies that there is a bound of the form $|\nabla u(z)|\le Mr $ for $dist(z,\partial \Omega_t)\le r$.   Thus if $|z-y|<c r$ we have
$$ u(z) \ge u(y) - M c r^2\ge (\epsilon -Mc)r^2$$
which is strictly positive as long as we take $c< M/\epsilon$.  Thus $B_{c r}(y)\subset \Omega_t$ proving \eqref{eq:measurezeroclaimII}.

So, letting $|A|$ denote the Lebesgue measure of a set $A$, this implies
$$ | B_r(x)\cap \partial \Omega_t| \le |B_r(x)| - |B_{c r}(y)|  = O((1-c^2) r^2).$$
Thus the Lebesgue density of $\partial \Omega$ at the point $z$ satisfies
$$ \delta(x):= \lim_{r\to 0} \frac{| B_r(x)\cap \partial \Omega_t|}{|B_r(x)|} <1.$$
But the Lebesgue Density Theorem \cite[Theorem 5.3.1]{Rudin} says $\delta(y)=1$ for almost all point $y\in \partial \Omega$, and thus $\partial \Omega$ must have measure zero as claimed.
\end{proof}

\begin{corollary}\label{cor:indofphi}
For all $t\in [0,1)$ it holds that
\begin{equation}
\omega_{\psi_t} = (1-\chi_{\Omega_t}) \omega +t\delta_{z_0}.\label{eq:laplaceenvelope2}
\end{equation}
In particular
\begin{equation}\int_{\Omega_t}\omega=t.\label{eq:volumeHSflow}\end{equation}
\end{corollary}
\begin{proof}
Since $\psi_t$ is $\mathcal{C}^{1,1}$, $\omega_{\psi_t}$ is absolutely continuous with respect to $\omega$, thus $\partial \Omega_t$ having zero measure with respect to $\omega$ means the same is true for $\omega_{\psi_t}$. We thus get $$\omega_{\psi_t}=\chi_{\Omega_t}\omega_{\psi_t}+(1-\chi_{\overline{\Omega}_t})\omega_{\psi_t}.$$
We have already seen $\omega_{\psi_t}=\delta_{z_0}$ on $\Omega_t$. By definition $\psi_t=0$ on $\Omega_t^c$ and hence on $(\Omega^c_t)^{\circ}$. As this set is open $dd^c \psi_t = 0$ there, giving $$(1-\chi_{\overline{\Omega}_t})\omega_{\psi_t}=(1-\chi_{\overline{\Omega}_t})\omega.$$ 
Again using that $\partial \Omega_t$ has zero measure yields $$(1-\chi_{\overline{\Omega}_t})\omega=(1-\chi_{{\Omega}_t})\omega$$ and hence \eqref{eq:laplaceenvelope2}. The second statement follows from this as
$$ \int_X \omega = \int_X \omega_{\psi_t} =\int_X ((1-\chi_{\Omega_t}) \omega +t\delta_{z_0})=  \int_{\Omega_t^c} \omega + t.$$
\end{proof}

We end this section with a final convexity property satisfied by the Hele-Shaw envelopes.  Although simple, it is essential in ensuring no information is lost when we later take the Legendre transform.

\begin{lemma}[Convexity]\label{lem:convexity}
For any given $z$ the function  $t\mapsto \psi_t(z)$ is concave, decreasing and continuous in $t$.
\end{lemma}

\begin{proof}
It is clear $\psi_t$ is concave in $t$ since if $t=at_1+(1-a)t_2$ where $a\in [0,1)$ and $t_1,t_2\in [0,1]$ then $$a\psi_{t_1}+(1-a)\psi_{t_2}\leq \psi_t$$ simply because the left hand side is clearly in $\Sh(X,\omega)$,  has at least Lelong number $t$ at $z_0$ and is bounded above by $0$.   That $\psi_t(z)$ decreases with $t$ is obvious, and this implies $\lim_{t \to s-}\psi_t$ is $\omega$-subharmonic and thus one sees $$\lim_{t \to s-}\psi_t=\psi_s,$$ i.e. $\psi_t$ is left-continuous in $t$.  Combined with concavity this implies continuity.  
\end{proof}

\subsection{Basic properties of the Hele-Shaw flow in the plane}\label{sec:basicplane}

We will also want to discuss the weak Hele-Shaw on the plane.  So suppose in this section $X=\mathbb C$ and $z_0$ is the origin.  Our K\"ahler metric $\omega$ can then be written as
$$ \omega = dd^c \phi$$
for some smooth function $\phi:\mathbb C\to \mathbb R$.  We assume throughout the growth condition that for all $t>0$
\begin{equation}\phi(z) \ge t\ln |z|^2 + O(1) \text{ for } |z|\gg 0\label{eq:growth}
\end{equation}
So, for example, this clearly holds for the standard K\"ahler metric on $\mathbb C$ for which $\phi(z) = |z|^2$.  

We are not assuming that the plane has finite area with respect to $\omega$, and so we need to add a word as to why the basic properties of the Hele-Shaw flow from the previous section continue to hold.  Given any $t>0$ consider the function
$$ \alpha_t = t\ln |z|^2 - \phi.$$
Clearly $\alpha_t\in \Sh(\mathbb C,\omega)$ and $\omega_{\alpha_t} = t\delta_0$ and $\nu_0(\alpha_t) = t$.  On the other hand the growth condition \eqref{eq:growth} implies $\alpha_t$ is bounded as $|z|$ tends to infinity, so subtracting a constant we may suppose $\alpha_t\le 0$.  Thus we may use $\alpha_t$ to replace the function provided  by Lemma \ref{lem:existenceofalphat}.   Using this one can check the proofs of the basic properties of the Hele-Shaw envelope go through essentially unchanged and give the following.

\begin{proposition}[Basic Properties of the weak Hele-Shaw flow in the Plane]\label{prop:basicplane}
Still assuming the growth condition \eqref{eq:growth} holds, for all $t>0$ we have (1) $\psi_t\in \Sh(\mathbb C,\omega)$ is locally bounded away from $z_0$ (2) $\nu_0(\psi_t)\ge t$ (3) $\psi_t\in \mathcal{C}^{1,1}(\mathbb C\setminus \{0\})$ (4) $\omega_{\psi_t} = (1-\chi_{\Omega_t}) \omega + t\delta_0$ (5) $\int_{\Omega_t} \omega = t$ (6) $\Omega_t$ is open, connected, contains the origin and $\partial \Omega_t$ has measure zero.  \end{proposition}

Furthermore,  analogs of the monotonicity and locality statements (Lemma \ref{lem:monotonicity} and Corollary \ref{cor:locality}) hold; precise statements are left to the reader.  Of course one can relate the planar case and the compact case by thinking of $\mathbb C\subset \mathbb P^1$ in the standard way.  Given any large $R$ one can find a K\"ahler form $\tilde{\omega}$ on $\mathbb P^1$ that agrees with $\omega$ on the ball $S: = \{ |z|<R\}$.  Then, with an argument as in the proof of the monotonicity statement (Lemma \ref{lem:monotonicity}) if one assumes the weak Hele-Shaw flow domains $\Omega_t$ and $\tilde{\Omega_t}$ induced by $\omega$ and $\tilde{\omega}$ respectively  are both are relatively compact in $S$, then $\Omega_t = \tilde{\Omega}_t$.     In this way one easily passes from statements about the Hele-Shaw on the plane to corresponding statements on $\mathbb P^1$.

\subsection{Bibliographical remarks}\label{sec:bibremarksHS}

For a much more comprehensive survey on the Hele-Shaw flow, which also goes under the name of Laplacian-growth, the reader is referred to the book of Gustafsson-Teodorescu-Vasil`ev \cite{Gustafssonbook} which also serves as a guide to the vast literature.      A difference between what is written here  is that we have been working on a compact Riemann surface endowed with a K\"ahler form, whereas the more classical treatment involves the complex plane, usually with the standard Euclidean structure. However this has little effect, and essentially all the fundamental results from the Hele-Shaw theory carry over without difficulty.     The point of view of the Hele-Shaw flow on Riemann surfaces was taken up by Hedenmalm-Shimorin \cite{Hedenmalm} and Hedenmalm-Olofsson \cite{Hedenmalm2}, who emphasise particularly the case of  simply connected Riemann surfaces, and it is from these papers that several of the basic properties above are taken.    The case of the flow on non-simply connected compact Riemann surfaces has been studied more recently by Skinner \cite{Skinner}.     

One can ask for more information about the structure of the boundary $\partial \Omega_t$.  At least when $X=\mathbb C$ and the background metric is real analytic, it is know this boundary consists of a finite number of real simple analytic curves having a finite number of double and cusp points (see \cite{CaffarelliR, Hedenmalm} as well as the work of Sakai \cite{Sakai4,Sakai2,Sakai3,Sakai,Sakai5}).
 
  Constructions similar to the Hele-Shaw envelope are abundant in (pluri)potential theory (as we have seen they show up in the Perron-Bremermann envelope and pluricomplex Green function) and sometimes go under the name of ``extremal envelopes" (see, for example \cite{ComanGuedj, Guedjcapacities, Larusson1,Larusson4, Larusson2, Larusson3,Rashkovskii1,Rashkovskii2, RashkovskiiSigurdsson1,RashkovskiiSigurdsson2}).

Lemma \ref{lem:alternativeenvelope} casts the Hele-Shaw envelope $\psi_t$ in the framework of variational inequalities and obstacle problems which is a subject in its own right (see, for instance \cite{CaffarelliK, Friedman, Kinderlehrer,Petrosyan}).   Perhaps the most important property of $\psi_t$ we have discussed is its $\mathcal{C}^{1,1}$ regularity (sometimes called ``optimal regularity"), from which we deduced both $\partial \Omega_t$ has measure zero and a formula for $\omega_{\psi_t}$ (in fact for this second statement, at least, one can get away with slightly less regularity).  Both of these results originate with the work of Caffarelli-Kinderlehrer \cite{CaffarelliK} and Caffarelli-Rivi\`ere \cite{CaffarelliR}  who restrict attention, for the most part, to domains in $\mathbb R^n$ (although given Lemma \ref{lem:genuineobstacle} it may well be possible that their techniques can be used to prove Theorem \ref{lemma:basicHS}).   Regularity of related envelopes, especially in higher dimensions, has been taken up in many places, for instance \cite{Berman2,Berman, BermanDemailly, ChuZhou, DarvasRubenstein, RossNystromEnvelopes,Tossatti}.

The radially symmetric case from Example \ref{ex:radiallysymmetric} can be generalised to toric manifolds, which was considered by Shiffman-Zelditch \cite{Shiffman} and Pokorny-Singer \cite{SingerPokorny}.  There appear to be many different names for the domain $\Omega_t$ and its complement.   In \cite{Shiffman} the analog of $\Omega_t$ is called the ``forbidden region".   The complement $X\setminus \Omega_t$ is called the ``equilbrium set" by Berman \cite{Berman} and in the theory of variational inequalities and obstacle problems $\partial \Omega_t$ sometimes goes under the name of ``free boundary" and $X\setminus \Omega_t$ goes under the name of ``coincidence set" (e.g.\ \cite[Definition 6.8]{Kinderlehrer}).

\section{The Duality Theorem}
We are now ready to connect the weak Hele-Shaw flow to the HMAE.  We continue with $X$ being a compact connected Riemann-surface with distinguished point $z_0$ and background K\"ahler form $\omega$ normalised so $\int_X\omega =1$. 

\subsection{Another HMAE} 

Let $\pi_X\colon X\times \overline{\mathbb D} \to X$ and $\pi_{\mathbb D}\colon X\times \overline{\mathbb D} \to \overline{\mathbb D}$ be the projections.  

\begin{definition}Set
\begin{equation}\label{eq:weaksolutionDtimes}
\tilde{\Phi} :=\sup \left\{ 
\begin{array}{c}
\Psi\in \Psh(X\times \overline{\mathbb D}, \pi_X^*\omega) \text{ : } \limsup_{\zeta\to \zeta'} \Psi(\zeta)\le 0\text{ for } \zeta'\in X\times \partial \mathbb D \\\quad\quad \text{ and } \nu_{(z_0,0)}(\Psi)\ge 1
\end{array}
\right\}.
\end{equation}\end{definition}

In the above, if $\tau$ is the standard coordinate on $\overline{\mathbb D}$ and $z$ a holomorphic coordinate on $X$ defined near $z_0$ the Lelong number condition $\nu_{(z_0,0)}(\Psi)\ge 1$ means that for all $c<1$,
$$ \Psi(z,\tau) \le c \ln ( |z-z_0|^2+|\tau|^2) + O(1) \text{ for } (z,\tau) \text{ near } (z_0,0).$$
Clearly $\tilde{\Phi}$ is analogous to the pluricomplex Green function discussed in \S \ref{sect:pluricomplex}.  The reason for us introducing this function is it is the weak solution for the following HMAE:

\begin{proposition}\label{prop:weakhmaetildephi}
The function $\tilde{\Phi}$ lies in $\Psh(X\times \overline{\mathbb D}, \pi_X^*\omega)$, is locally bounded away from $(z_0,0)$ and solves
\begin{align}
%\pi_X^*\omega + dd^c \tilde{\Phi} & \ge 0 \label{eq:hmaelelong1}\\
\MA_{\pi_X^*\omega}(\tilde{\Phi}) &=0 \text{ on } X\times \mathbb D \setminus \{(z_0,0)\} \label{eq:hmaelelong1b}\\
\nu_{(z_0,0)}(\tilde{\Phi}) &\ge 1 \label{eq:hmaelelong2}  \\
\lim_{|\tau|\to 1} \tilde{\Phi}(z,\tau) &= 0.\label{eq:continuityboundary}
%\tilde{\Phi} &= 0 \text{ over } X\times \partial\overline{\mathbb D} \label{eq:hmaelelong3}
\end{align}
Furthermore
\begin{equation}\tilde{\Phi}(z,\tau) = \tilde{\Phi}(z,e^{i\theta} \tau) \text{ for all } (z,\tau)\in X\times\overline{\mathbb D} \text{ and } \theta\in \mathbb R.\label{eq:hmaelelong4}
\end{equation}
\end{proposition}

\begin{proof}
We give a sketch proof.  Observe first both $\ln |\tau|^2$ and the function $\alpha(z)$ from Lemma \ref{lem:existenceofalphat} are candidates for the envelope defining $\tilde{\Phi}$, which implies it is locally bounded away from $(z_0,0)$ and $\ln |\tau|^2\le \tilde{\Phi}(z,\tau)$.  On the other hand the maximum principle applied to the slices $\{z\}\times \overline{\mathbb D}$ shows that $\tilde{\Phi}\le 0$ over $X\times \overline{\mathbb D}$, giving \eqref{eq:continuityboundary} 

For \eqref{eq:hmaelelong2} it is convenient to consider the blowup $p: Y\to X\times \overline{\mathbb D}$ at the point $(z_0,0)$ which has an exceptional divisor we denote by $E$.  Suppose $\pi_X^*\omega + dd^c\Psi \ge 0$ satisfies $\Psi \le c( \ln |z-z_0|^2 + |\tau|^2) + O(1)$ near $(z_0,0)$.  Then $E$ is covered by open subsets $U$ on which $E$ is the zero set of some holomorphic function $u$ say, so that $p^* \Psi|_U \le c \ln |u|^2 + O(1)$.  Then similar to the proof of Proposition \ref{prop:HSlelong}(3c), one can use the maximum principle to deduce in fact $p^*\Psi|_U \le c\ln |u|^2 + O(1)$ (and thus  $\Psi\le c\ln (|z-z_0|^2 + |\tau|^2) +O(1)$) for an $O(1)$ term that is independent of $\Psi$.  We leave the details to the reader.

The fact that $\tilde{\Phi}$ solves the claimed HMAE is as in the Perron-Bremermann envelope.    Finally \eqref{eq:hmaelelong4} is a consequence of the previous statements, since if $\theta$ is fixed then $\tilde{\Phi}(z,e^{i\theta} \tau)$ is a candidate for the envelope defining $\tilde{\Phi}$.
\end{proof}

It is convenient to extend the domain of definition of $\tilde{\Phi}$.  By \eqref{eq:hmaelelong4} for fixed $z\in X$, the function $\tilde{\Phi}(z,e^{-s/2})$ is independent of the imaginary part of $s$, and is subharmonic.  Thus we can think of  $\tilde{\Phi}(z,e^{-s/2})$ as a convex function of $s\in [0,\infty)$.   If we set
$$\tilde{\Phi}(z,e^{-s/2}) = +\infty \text{ for } s<0$$
then $\tilde{\Phi}(z,e^{-s/2})$ is a convex function for all $s\in \mathbb R$.

\begin{theorem}[Duality Theorem,  Ross-Witt Nystr\"om \cite{RWHarmonicDiscs}]
The weak solution $\tilde{\Phi}(z,\tau)$ to the HMAE and the Hele-Shaw envelopes $\psi_t(z)$ are related by a Legendre transform. That is, 
\begin{equation} \label{legendre1}
\psi_t(z)=\inf_{|\tau|>0}\{\tilde{\Phi}(z,\tau)-(1-t)\ln|\tau|^2\}
\end{equation} 
and 
\begin{equation} \label{legendre2}
\tilde{\Phi}(z,\tau)=\sup_{t}\{\psi_{t}(z)+(1-t)\ln |\tau|^2\}.
\end{equation}
\end{theorem}

\begin{proof}
For $t\in [0,1]$ consider 
$$\alpha_t(z,\tau):=\psi_{t}(z)+(1-t)\ln |\tau|^2 \text{ for } (z,\tau)\in X\times \overline{\mathbb D}.$$ 
Clearly $\alpha_t\le 0$ and $\pi_X^* \omega + dd^c \alpha_t = \pi_X^* \omega_{\psi_t}\ge 0$.  Also as $\nu_{z_0}(\psi_t)\ge t$, 
$$\alpha_t (z,\tau)\le t\ln |z-z_0|^2 + (1-t) \ln |\tau|^2 + O(1) \le \ln ( |z-z_0|^2 + |\tau|^2) + O(1).$$
Thus $\alpha_t$ is a candidate for the envelope defining $\tilde{\Phi}$ giving
\begin{equation}
 \psi_t(z) \le \tilde{\Phi}(z,\tau) - (1-t)\ln |\tau|^2.\label{eq:duality1}
 \end{equation}
On the other hand if $t>1$ then $\psi_t\equiv -\infty$ and if $t<0$ then $\tilde{\Phi}(z,\tau) - (1-t)\ln |\tau |^2 \ge \tilde{\Phi}(z,\tau) - \ln |\tau|^2 \ge 0 = \psi_t(z)$.  Hence  \eqref{eq:duality1} holds for all $t\in \mathbb R$, and taking the infimum over all $|\tau|>0$,
$$\psi_t(z)\le \inf_{|\tau|>0}\{\tilde{\Phi}(z,\tau)-(1-t)\ln|\tau|^2\}.$$
For the other inequality, since $\tilde{\Phi}(z,\tau)$ is independent of the argument of $\tau$, it follows from Kiselman's minimum principle \cite{Kiselman} that 
$$\tilde{\psi}_{t}(z): =  \inf_{|\tau|>0}\{\tilde{\Phi}(z,\tau)-(1-t)\ln|\tau|^2\}$$
is in $\Sh(X,\omega)$.  We wish to show $\tilde{\psi}_t$ is a candidate for the envelope defining $\psi_t$.  

First, using \eqref{eq:continuityboundary} and letting $\tau\to 1$ gives $\tilde{\psi}_t \le 0$.  We claim $\nu_{z_0}(\tilde{\psi}_t)\ge t$.  To see this, recall we are thinking of $\tilde{\Phi}(z,e^{-s/2})$ as a convex function in $s\in [0,\infty)$.  So for a fixed $z$
\begin{equation}
\tilde{\psi}_t(z)=\inf_{s\geq 0}\{\tilde{\Phi}(z,e^{-s/2})+(1-t)s\}.\label{eq:dualityeq1}
\end{equation}
Now $\tilde{\Phi}$ has Lelong number at least 1 at $(z_0,0)$.  So for any $c<1$ there is a constant $C$ such that
\begin{equation}
\tilde{\Phi}(z,\tau) \le c\ln (|z-z_0|^2 + |\tau|^2) +C =c \ln(|z-z_0|^2 +e^{-s}) +C\label{eq:lelongtildepsi}
\end{equation} 
for $(z,\tau)$ near $(z_0,0)$.  Combining with  \eqref{eq:dualityeq1} yields \begin{equation}
\tilde{\psi}_{t}(z)\leq c\inf_{s\geq 0}\{\ln (e^{-s}+|z-z_0|^2)+(1-t)s\}+C.\label{eq:dualityeq2}
\end{equation}

By elementary means one easily checks if $t\in (0,1)$ the infimum of $\ln (e^{-s}+|z-z_0|^2)+(1-t)s$ is attained when $e^{-s} = \frac{1-t}{t}|z-z_0|^2$ and at this point the right hand side of \eqref{eq:dualityeq2} is equal to
$$c(t\ln |z-z_0|^2 -(1-t)\ln (1-t) -t\ln t)+C.$$
Hence $\tilde{\psi}_t(z)\le ct\ln |z-z_0|^2 + O(1)$ for $z$ near $z_0$.   Since this holds for all $c<1$ we conclude $\nu_{z_0}(\tilde{\psi}_t)\ge t$ for $t\in (0,1)$.   For $t=0$ one notes $\ln |\tau|^2$ is a candidate for the envelope defining $\tilde{\Phi}$, so $ \ln |\tau|^2\le \tilde{\Phi}$, which gives $\tilde{\psi}_0 = \inf_{s\ge 0}\{ \tilde{\Phi}(z,e^{-s/2}) + s\} \ge 0$ and hence in fact $\tilde{\psi}_0 = 0 = \psi_0$.  For $t=1$, observe $\tilde{\psi}_1(z)\leq \tilde{\Phi}(z,0) \le \ln |z-z_0|^2 + O(1)$  so $\nu_{z_0}(\tilde{\psi}_t) \ge 1$.  For $t<0$ then certainly $\nu_{z_0}(\tilde{\psi}_t) \ge t$, and thus we conclude for $t\le 1$ that $\tilde{\psi}_t$ is a candidate for the envelope defining $\psi_t$, and thus $\tilde{\psi}_t=\psi_t$.  Finally for $t>1$ by taking $s\to \infty$ in the definition of $\tilde{\psi}_t$ it is immediate $\tilde{\psi}_t\equiv -\infty$ and so $\tilde{\psi}_t = \psi_t$ for all $t$ giving (\ref{legendre1}).

After some rearranging, we have shown  $$-\psi_t(z)=\sup_{s \in \mathbb{R}}\{ts-(\tilde{\Phi}(z,e^{-s/2})+s)\},$$ i.e. that $-\psi_t(z)$ is the Legendre transform of $u(s):=\tilde{\Phi}(z,e^{-s/2})+s.$   So, the second statement follows from the first by the involution property of the Legendre transform.  In fact, we can see that $u(s)$ is convex and lower semicontinuous (since it is continuous on $[0,\infty)$ and constantly $-\infty$ on $(-\infty,0)$).  Thus 
by the Fenchel-Moreau Theorem (see e.g. \cite{Rockafellar})  $u(s)$ is the Legendre transform of $-\psi_t(z)$ which is \eqref{legendre2}). 
\end{proof}

\begin{remark}
Thinking of $\tilde{\Phi}$ as a solution to the HMAE over the punctured disc, we can interpret it as a weak geodesic ray in the space of K\"ahler potentials $\mathcal K(X,\omega)$.  We have seen that $\tilde{\Phi}(z,\tau)$ depends only on the absolute value of $\tau$, and so using the variable $s=-\log |\tau|$ the potentials $\tilde{\Phi}(z,e^{-s})$ for $s\in [0,\infty)$ give a weak geodesic ray in this space, starting at the potential that is identically zero when $s=0$.  In the limit as $s\to \infty$, this ray ends up with a singular potential on $X$ that puts all of its mass at the distinguished point $z_0$ (and so it is these geodesic rays that are related through a Legendre transform to the Hele-Shaw flow).  In  previous work of Donaldson \cite{DonaldsonNahm}, a different free boundary problem is related, again through a Legendre transform, to the HMAE over the annulus, and thus to weak geodesic segments in $\mathcal K(X,\omega)$.
\end{remark}

\subsection{Connection with the Hele-Shaw domains}

So far we have related the solution $\tilde{\Phi}$ to the HMAE with the Hele-Shaw envelopes, and now we connect it to the weak Hele-Shaw domains. 

\begin{definition}\label{def:hamiltonian}
Let $H\colon X\times \overline{\mathbb{D}}^{\times}\to \mathbb R$ be defined by
\begin{equation}
H(z,\tau):=\frac{\partial}{\partial s^+} \tilde{\Phi}(z,e^{-s/2})\label{eq:defhamiltonian}
\end{equation}
where $s:=-\ln |\tau|^2$.  %When $|\tau|=1$ (and thus $s=0$) we interpret this to mean the the right derivative.  %We refer to $H$ as the \emph{Hamiltonian} associated to $(X,\omega)$.  
\end{definition}
Here the notation means we are taking the right derivative, which by by convexity of $s\mapsto H(z,e^{-s/2})$ always exists.  Our reason for introducing this function is that it records the time at which the weak Hele-Shaw flow arrives at a given point in $X$.

\begin{proposition}\label{prop:Hu_t}
$$H(z,1)+1=\sup\{t: \psi_{t}(z)=0\} = \sup\{ t: z\notin \Omega_t\}.$$
\end{proposition}
\begin{proof}
From (\ref{legendre2}) if $\psi_{t}(z)=0$ then $$\tilde{\Phi}(z,e^{-s/2})\geq (t-1)s$$ 
where as always $s=-\ln |\tau|^2$, and thus by convexity $$H(z,1)\geq t-1.$$   For the other direction, suppose $\psi_{t}(z)=a$ for some $a<0$.  Recalling for a fixed $z$ the function $t'\mapsto \psi_{t'}(z)$ is concave and decreasing in $t'$, one sees that for $t\le t'\le 1$ and $s\ge 0$ we have $\psi_{t'}(z) + (t'-1)s \le a$.  On the other hand $\psi_{t'}\le 0$ so if $0\le t'\le t$ then $\psi_{t'}(z) + (t'-1)s\le (t-1)s$.  Putting this together with (\ref{legendre2}) gives $$\tilde{\Phi}(z,e^{-s/2})\leq  \max( (t-1)s,a)$$ and so $H(z,1)\leq t-1,$ which proves the proposition.
\end{proof}

As an application we are able to give the following statement about the movement of the boundary of the Hele-Shaw flow.   By means of notation, for any $S\subset X$ and $r>0$ let
$$ S + B_r = \{ z\in X :  d(z,z')<r \text{ for some } z'\in S\}$$
where $d$ denotes a fixed distance function on  $X$ (for instance we could take the geodesic
distance with respect to the background K\"ahler metric determined by $\omega$).

\begin{corollary}\label{cor:HSmovement}\ 
\begin{enumerate}
\item Assume that  $H(\cdot, 1)$ is continuous.  Then the boundary of the weak Hele-Shaw flow is strictly increasing.  That is, if $z\in \partial \Omega_t$ for some $t>0$ then $z\in \Omega_{t'}$ for all $t'>t$.
\item Assume that $H(\cdot, 1)$ is moreover Lipschitz.  Then there is a lower bound on the rate of increase of the weak Hele-Shaw flow.  That is, there exist a $\delta>0$ such that for all $0<t<t'<1$ 
$$\Omega_t + B_{\delta(t-t')} \subset \Omega_{t'}.$$
\end{enumerate}
\end{corollary}
\begin{proof}
We start with the first statement. Let $z\in \partial \Omega_t$ and $(z_n)_{n\in \mathbb N}$ be a sequence of points in $\Omega_t$ tending to $z$ as $n$ tends to infinity.  Fixing $n$ we then have $z_n\in \Omega_s$ for all $s\ge t$ and so
$$ H(z_n,1) + 1 \le t.$$
By continuity of $H(\cdot,1)$ this implies
$$ H(z,1) + 1 \le t$$
and so if $t'>t$ we must have $z\in \Omega_{t'}$ as desired.

For the second statement, let $C$ be the Lipschitz constant of $H(\cdot, 1)$, so
$$ |H(z,1) - H(\tilde{z},1)|\le C d(z,\tilde{z}) \text{ for all } z,\tilde{z}\in X,$$
and set $\delta = C^{-1}$.  Fix $t'>t$ and $z\in \Omega_t + B_{\delta (t'-t)}$.  Then there exists $z'\in \Omega_t$ with $d(z,z')<\delta (t'-t)$.    As $z'\in \Omega_t$ we clearly have $H(z')+1\le t$.  On the other hand if $z\notin \Omega_{t'}$ then $H(z)+1\ge t'$ giving
$$ t'-t \le H(z) - H(z') \le Cd(z,z')  < C\delta (t'-t) =  t'-t$$
which is absurd.  Hence we must have $z\in \Omega_{t'}$ as required.
\end{proof}

\begin{remark}
Of course if $\tilde{\Phi}$ lies in $\mathcal{C}^{1,1}(X\times \Sigma)$ then $H$ will be Lipschitz.  We will see in the next section that this always holds when $X=\mathbb P^1$, and expect that $\tilde{\Phi}$ should be at least $\mathcal{C}^{1,\alpha}$ for all $\alpha<1$ when $X$ is a general compact Riemann surface.  

Even in the case when $X=\mathbb P^1$ Corollary \ref{cor:HSmovement} is new (as far as we are aware).  Heden\-malm-Shimorin have a similar statement \cite[Proposition 3.2]{Hedenmalm} but under the hypothesis $\Omega_t$ is simply connected along with some regularity assumptions about $\partial \Omega_t$.    The proof above rests on regularity of $\tilde{\Phi}$, and it would be interesting to compare this with a proof (if one exists) that uses only one-dimensional techniques such as those from \S \ref{sec:HS}.  
\end{remark}

\subsection{Twisting}\label{sec:hmaetwist}

We end this section by discussing a certain ``twisting" technique that applies when $X=\mathbb P^1$ to show the quantity $\tilde{\Phi}$
%$$\tilde{\Phi} :=\sup \left\{
%\begin{array}{c}
%\psi \colon \mathbb P^1\times \overline{\mathbb D}\to \mathbb R\cup \{-\infty\} : \psi\text{ is usc, } \pi_{\mathbb P^1}^*\omega + dd^c\psi\ge 0 \\\text{and } \psi(z,\tau)\le 0 \text{ for } (z,\tau)\in \mathbb P^1\times \partial \mathbb D \text{ and } \nu_{(z_0,0)}(\psi)\ge 1
%\end{array}\right\}
%$$
we have been considering can be expressed in a different way without the condition on the Lelong number.    We have two motivations for wanting to do this.  First, the new formulation solves the classical version of the HMAE as discussed in the introduction, and thus this twisting relates it also to the Hele-Shaw flow.   Second, we can use known regularity results about this version of the HMAE to conclude regularity of $\tilde{\Phi}$.   
 
 The necessity of restricting to $\mathbb P^1$ is that we will make use of the existence of a global holomorphic $S^1$-action.  Consider $\mathbb P^1$ covered by two copies of $\mathbb C$ in the standard way with coordinates $z$ and $w=1/z$.   For non-zero $\tau\in \mathbb \overline{\mathbb D}$ the map
 $$\rho_{\tau}:\mathbb P^1\to \mathbb P^1 \text{ given by } f(z) = \tau z$$
 is a biholomorphism fixing $z_0$.  Restricting to those $\rho_{\tau}$ with $|\tau|=1$ gives a global holomorphic $S^1$-action.
 
 Now $\rho_{\tau}^*\omega$ lies in the same cohomology class as $\omega$ and hence we can write
 $$ \rho^*_\tau \omega = \omega + dd^c\phi_\tau$$
 for some smooth function $\phi_{\tau}$ on $X$.  By choosing these to be normalised by requiring $\int_X \phi_\tau \omega =0$, the $\phi_\tau$ are uniquely defined and 
 \begin{equation}\label{eq:boundarydisc}
 \phi(z,\tau) : = \phi_\tau(z)
 \end{equation}
 is a smooth function on $X\times \overline{\mathbb D}^\times$.

 From now on,  let $\Phi$ be the Perron-Bremermann envelope on $X\times \overline{\mathbb D}$ with boundary data $\phi$.  Thus
  \begin{equation}\label{eq:weaksolutionD}
\Phi :=\sup \left\{ 
\begin{array}{c}
\Psi\in \Psh(X\times \overline{\mathbb D}, \pi_X^*\omega) \text{ : } \limsup_{\zeta\to \zeta'} \Psi(\zeta)\le \phi(\zeta')\text{ for } \zeta'\in X\times \partial \mathbb D 
\end{array}
\right\}.
\end{equation}

 So the difference between $\tilde{\Phi}$ and $\Phi$ is the latter is taken with respect to the ``twisted" boundary data $(z,\tau)\mapsto \phi( z,\tau)$, but does not have any condition on the Lelong number at $(z_0,0)$.  The following simple Lemma gives the explicit relationship between these two envelopes.  It will be crucial later on when we wish to translate results about envelopes over the punctured disc (which connects most naturally with the Hele-Shaw flow on $X$) to analogous statements about envelopes over the unpunctured disc.

\begin{lemma} \label{lem:twist}
We have
%$$\Phi(z,\tau)+\phi(z)+\ln|\tau|^2+\ln(1+|z|^2)=\tilde{\Phi}(\tau z,\tau)+\phi(\tau z) + \ln(1+|\tau z|^2) \text{ for } (z,\tau)\in \mathbb P^1\times \overline{\mathbb D}^{\times}.$$
$$\Phi(z,\tau)=\tilde{\Phi}(\tau z,\tau)+\phi( z,\tau) -\ln |\tau|^2 \text{ for } (z,\tau)\in \mathbb P^1\times \overline{\mathbb D}^{\times}.$$
\end{lemma}
\begin{proof}
Let $\beta(z,\tau):=\Phi(\tau ^{-1} z,\tau)-\phi(\tau^{-1} z,\tau) +\ln |\tau|^2$.  One easily checks if $|\tau|=1$ then $\beta(z,\tau)=0$ and $\pi_{\mathbb P^1}^* \omega + dd^c\beta \ge 0$ and also $\nu_{(z_0,0)}(\beta)\ge 1$.  Hence $\beta(z,\tau) \le \tilde{\Phi}(z,\tau)$ giving one inequality, and the other is proved similarly.
\end{proof}

\begin{theorem}\label{thm:c11}
When $X=\mathbb P^1$ the envelope $\tilde{\Phi}$ is $\mathcal{C}^{1,1}$ on $\mathbb P^1\times \overline{\mathbb D}^\times$.
\end{theorem}
\begin{proof}
From the work of Chu-Tossati-Weinkove (Theorem \ref{thm:C11}) we have $\Phi$ is $
\mathcal{C}^{1,1}$ over $X\times \overline{\mathbb D}$ (we could also use the work of B\l ocki \cite{Blocki2} as $\mathbb P^1$ has nonnegative bisectional curvature so \cite[Theorem 1.4]{Blocki2} applies).  Thus the desired statement for $\tilde{\Phi}$ follows from Lemma \ref{lem:twist}.  
\end{proof}

\begin{remark}\label{rem:C1alphaingeneral}
It seems likely on a general compact Riemann surface that $\Phi$ also satisfies some regularity, and should be at least $\mathcal{C}^{1,\alpha}$ for any $\alpha<1$.    Our reason for saying this is $\Phi$ is describing a weak geodesic ray in the space of K\"ahler potentials on $X$, and such regularity is known to hold for many related geodesic rays, such as those considered by Phong-Sturm \cite{PhongSturmRegularity}.
\end{remark}

\begin{remark}
A point $c$ lying on the boundary $\partial \Omega_t$ of the Hele-Shaw domain for $t$ in some non-trivial interval is referred to as a \emph{stationary point}.  Theorem \ref{thm:c11} combined with Corollary \ref{cor:HSmovement}(1) imply that the Hele-Shaw flow on $\mathbb P^1$ with a smooth area form and empty initial condition never develops any stationary points (as far as we are aware this statement in the smooth case is new).
\end{remark}

\section{Harmonic discs}

We return now to the case of a general compact Riemann surface $X$ with K\"ahler form $\omega$.  The next theorem describes precisely the proper harmonic discs of the weak solution to the HMAE in terms of the Riemann map of those weak Hele-Shaw domains that are simply connected.

As above consider
\begin{equation}
\tilde{\Phi} :=\sup \left\{ 
\begin{array}{c}
\Psi\in \Psh(X\times \overline{\mathbb D}, \pi_X^*\omega) \text{ : } \limsup_{\zeta\to \zeta'} \Psi(\zeta)\le 0\text{ for } \zeta'\in X\times \partial \mathbb D \\\quad\quad \text{ and } \nu_{(z_0,0)}(\Psi)\ge 1
\end{array}
\right\}.
\end{equation}

%$$\tilde{\Phi} :=\sup \left\{
%\begin{array}{c}
%\psi \colon X\times \overline{\mathbb D}\to \mathbb R\cup \{-\infty\} : \psi\text{ is usc, } \pi_X^*\omega + dd^c\psi\ge 0 \\\text{and } \psi(z,\tau)\le 0 \text{ for } (z,\tau)\in X\times \partial \mathbb D \text{ and } \nu_{(z_0,0)}(\psi)\ge 1
%\end{array}\right\}.
%$$

\begin{definition}
We say the graph of a holomorphic $f:\mathbb D\to X$ is a proper \emph{harmonic disc} for $\tilde{\Phi}$ if  $\pi_X^*\omega + dd^c\tilde{\Phi}$ vanishes along the graph of $f$ away from the origin, or said another way
$\tilde{\Phi}(f(\tau),\tau)$ is $f^*\omega$-harmonic on $\mathbb D^{\times}$.
\end{definition}

%We assume $\tilde{\Phi}$ is in $\mathcal{C}^{1}(X\times \overline{\mathbb D})$ so the function
%$$H(z,\tau):=\frac{\partial}{\partial s} \tilde{\Phi}(z,e^{-s/2}) \text{ for } (z,\tau) \in X\times \overline{\mathbb D}^{\times}$$
%from \eqref{eq:defhamiltonian} is well-defined.

\begin{theorem}[Regularity Theorem, Ross-Witt Nystr\"om \cite{RWHarmonicDiscs}] \label{thm:holdiscs}
The graph of a holomorphic $f:\mathbb D\to X$ is a proper harmonic disc of $\tilde{\Phi}$ if and only if either\begin{enumerate}
\item $f$ is the constant map $f(\tau)=z_0$ for all $\tau \in \mathbb D$ (where $z_0$ is our given distinguished point in $X$)
\item For some $t$ the weak Hele-Shaw domain $\Omega_t$ for $\omega$ is simply connected and $f:\mathbb D\to \Omega_t$ is a biholomorphism (i.e.\ a Riemann map) with $f(0) = z_0$.
\item $f$ is the constant map $f(\tau)= z'$ for all $\tau\in \mathbb D$, for some fixed $z'\in X\setminus \Omega_1$.
\end{enumerate}
Moreover in the first case $H(f(\tau),\tau)\equiv -1$, in the second case $H(f(\tau),\tau)\equiv t-1$ and in the third $H(f(\tau),\tau)\equiv 0$.  
\end{theorem}

\begin{remark}
More generally we would say that a proper holomorphic curve $g:\Sigma \to X\times \mathbb{D}$ is a proper harmonic curve of $\tilde{\Phi}$ if $\tilde{\Phi}\circ g$ was $(\pi_X\circ g)^*\omega$ harmonic except at $g^{-1}(z_0,0)$. But it is not hard to see that any such $g$ would have to be a composition of one of the proper harmonic discs described in Theorem \ref{thm:holdiscs} with a finite cover of the unit disc, so in particular having the same image.
\end{remark}

Before the proof we need the following statement:

\begin{lemma} \label{lem:ham2}
Fix $0<|\tau|<1$.  Then $$H(z,\tau)=t-1 \qquad \Longleftrightarrow \qquad \tilde{\Phi}(z,\tau)=\psi_t(z)+(1-t)\ln |\tau|^2.$$
\end{lemma}

\begin{proof}
Fix a point $z\in X$ and $0<|\tau_0|<1$ and let $s_0 = -\ln |\tau_0|^2$.   From the Duality Theorem (\ref{legendre2})  $$\tilde{\Phi}(z,\tau_0)=\sup_{t}\{\psi_t(z)+(1-t)\ln |\tau_0|^2\}.$$    Now $\psi_t(z)$ is continuous in $t$ (Lemma \ref{lem:convexity}), so for some $t$
$$\tilde{\Phi}(z,\tau_0)=\psi_{t}(z)+(1-t)\ln |\tau_0|^2 = \psi_{t}(z) - (1-t)s_0.$$
On the other hand, we certainly have
$$\tilde{\Phi}(z,e^{-s/2})\geq \psi_{t}(z)-(1-t)s\text{ for all } s.$$
So the slope of the convex function $\tilde{\Phi}(z,e^{-s/2})$  at the point $(s_0,\tilde{\Phi}(z,e^{-s_0/2}))$ is equal to the slope of the linear function $s\mapsto \psi_t(z) - (1-t)s$, which is clearly  $t-1$.   Hence
$$H(z,\tau_0)=\frac{\partial}{\partial s^+}_{|s=s_0} \tilde{\Phi}(z,e^{-s/2})=t-1,$$
which is enough to prove the lemma.
\end{proof}

\begin{proof}[Proof of Theorem \ref{thm:holdiscs}]
 
We shall prove if the graph of $f$ is a proper harmonic disc for $\tilde{\Phi}$ then it is of one of the three forms in the statement of the theorem.  Fix some $\tau_0\in \mathbb D^\times$ and set
$$ t_0 = H(f(\tau_0),\tau_0) +1.$$
We claim
\begin{equation}
 \psi_{t_0}(f(\tau)) + (1-t_0) \ln |\tau|^2 = \tilde{\Phi}(f(\tau),\tau) \text{ for all }\tau\in \mathbb D. \label{eq:thmregeq1}
\end{equation}
To see this, consider
$$ \alpha(\tau): = \psi_{t_0}(f(\tau)) + (1-t_0) \ln |\tau|^2 - \tilde{\Phi}(f(\tau),\tau) \text{ for }\tau\in \mathbb D.$$
Then $\alpha$ is subharmonic (since $\tilde{\Phi}$ is $\pi_X^*\omega$-harmonic along $\{(f(\tau),\tau): \tau\neq 0\}$), satisfies $\alpha\le 0$  by the Duality Theorem \eqref{legendre2} and $\alpha(\tau_0) = 0$ by Lemma \ref{lem:ham2}.  If $f(0)\neq  0$ then  $\tilde{\Phi}$ is $\pi_X^*\omega$-harmonic even over   $\{(f(\tau),\tau): \tau\in \mathbb D\}$ and so \eqref{eq:thmregeq1} follows from the maximum principle.  If $f(0)=0$ then by looking at the Lelong number, $\alpha$ extends over $\tau=0$ and the maximum principle still applies to give \eqref{eq:thmregeq1}.   In particular Lemma \ref{lem:ham2} combined with \eqref{eq:thmregeq1} implies  $H(f(\tau),\tau)\equiv t_0-1$ for all $\tau\neq 0$, giving the last statement of the theorem.

Suppose now that $f$ is non-constant.  We shall show $f$ is as in case (2) of the statement, by first proving the image of $f$ lies in $\Omega_{t_0}$ and then proving it is a biholomorphism taking $0$ to $z_0$.     Observe if $\tilde{\Phi}(f(\tau),\tau)$ is $f^*\pi_X^*\omega$ harmonic on a neighbourhood of some $\tau'\in \mathbb D$, then \eqref{eq:thmregeq1} implies $\psi_{t_0}$ is $\omega$-harmonic on a neighbourhood of $f(\tau')$.  But  Corollary \ref{cor:indofphi} implies
\begin{equation}
\omega_{\psi_t} = (1-\chi_{\Omega_t}) \omega +t\delta_{z_0}.\label{eq:laplaceenvelope}
\end{equation}
so this in turn implies $f(\tau')\in \Omega_{t_0}$.    

By hypothesis, $\tilde{\Phi}(f(\tau),\tau)$ is $f^*\pi_X^*\omega$ harmonic on a neighbourhood of any non-zero $\tau'\in \mathbb D$, so $f(\tau')\in \Omega_{t_0}$ for all $\tau'\neq 0$.  In particular $\Omega_{t_0}$ is non-empty, so we must have $t_0>0$ and so $z_0\in \Omega_{t_0}$ by Corollary \ref{cor:open}.    If $f(0)=z_0$ then $f(0)\in \Omega_{t_0}$ .  On the other hand, if $f(0)\neq z_0$ then  $\tilde{\Phi}(f(\tau),\tau)$ is $f^*\pi_X^*\omega$-harmonic on a neighbourhood of $0\in \mathbb D$, giving $f(0)\in \Omega_{t_0}$.  Thus in either case  $f(0)\in \Omega_{t_0}$, and hence the image of $f$ lies in $\Omega_{t_0}$ as claimed.

We next prove $f$ is proper.   To see this let $\tau_i$ be a sequence in $\mathbb D$ such that $|\tau_i|\to 1$ as $i\to \infty$.  Then by \eqref{eq:thmregeq1} and then \eqref{eq:continuityboundary} 
$$ \lim_{i\to \infty} \psi_{t_0}(f(\tau_i),\tau_i)=\lim_{i\to \infty} \tilde{\Phi}(f(\tau_i),\tau_i) - (1-t_0)\ln |\tau_i| = 0.$$
But $\Omega_{t_0}$ is exhausted by the compact sets $\{z : \psi_{t_0}(z)\le  - 1/n\}$ for $n\in \mathbb N$ and $f(\tau_i)$ escapes to infinity in $\Omega_{t_0}$.  Thus $f$ is proper as claimed.

Next we show the preimage  $S:=f^{-1}(z_0)$ is precisely the point $0$ with multiplicity one.  Given this, the fact that $f$ is a biholomorphism with $f(0) = z_0$ follows from a standard argument with the winding number (Lemma \ref{lem:winding}).   Observe $f(\tau)\neq z_0$ for any $\tau\neq 0$, since otherwise the right hand side of \eqref{eq:thmregeq1} would be $-\infty$  whereas the left hand side is finite.  If $f(0)\neq z_0$ then $S$ would be empty, which is absurd by Lemma \ref{lem:winding}.  So we conclude $z_0\in S$ with some multiplicity $m\ge 1$.  Using \eqref{eq:thmregeq1} once again
$$ \psi_{t_0}(f(\tau)) + (1-t_0) \ln |\tau|^2 = \tilde{\Phi}(f(\tau),\tau) \ge  \ln |\tau|^2.$$
Clearly  $\psi_{t_0}(f(\tau))$ has Lelong number $mt_0$ at $0$, so the left hand side has Lelong number $mt_0 + (1-t_0)$ at $0$.  By the right hand side has Lelong number 1, giving $mt_0 + (1-t_0) \le 1$, so $m=1$.      So in conclusion we have shown if $f$ is not-constant then $f$ is of the form case (2). 

  Suppose now $f\equiv z'$ is constant.  If $z'=z_0$ then $f$ is as in case (1).  Otherwise $z'\neq z_0$, and so $\tilde{\Phi}(f(\tau),\tau)$ is $f^*\pi_X^*\omega$-harmonic even near $\tau=0$.  Again \eqref{eq:thmregeq1} gives
$$ \tilde{\Phi}(z',\tau) = \psi_{t_0}(z') + (1-t_0) \ln |\tau|^2 \text{ for } \tau \in \mathbb D.$$
But this implies $t_0=1$, else otherwise the right hand side takes the value $-\infty$ at the point $\tau=0$, whereas the left hand side is finite.   Letting $\tau\to 1$ and using \eqref{eq:continuityboundary}
$$\psi_1(z') = \lim_{\tau\to 1} \tilde{\Phi}(z',\tau) = 0$$
and hence $z'\in \Omega_1$, implying $f$ is as in case (3).

The converse, namely that each of the three listed functions, are proper harmonic discs is easier and is left to the reader.
\end{proof}

\begin{lemma} \label{lem:winding}
If $f\colon D_1\to D_2$ is a proper holomorphic map between two open domains in $\mathbb{P}^1$ then the number of preimages $N_p:=\#\{f^{-1}(p)\}$ (counted with multiplicity) is constant. 
\end{lemma}

\begin{proof}
Let $\gamma$ be a smooth curve in $D_2$ connecting two points $p$ and $q$ and let $U$ be a finite union of open discs compactly supported in $D_1$ which together cover the compact set $f^{-1}(\gamma)$. Since the image of any boundary component of $U$ cannot cross $\gamma$  the winding numbers of the image of any such boundary component with respect to $p$ and $q$ must be the same.  Since that winding number counts the number of preimages inside that component we get by adding up the winding numbers for the different boundary components that $N_p=N_q$.
\end{proof}

From this we get a description of all the proper harmonic discs for a more classical version of the HMAE, at least when $X=\mathbb P^1$.

\begin{corollary}
Let $X=\mathbb P^1$.  Then the graph of $g:\mathbb D\to \mathbb P^1$ is a proper harmonic disc for the weak solution to the HMAE over $X\times \overline{\mathbb D}$ with boundary data $\phi(z,\tau)$ from \eqref{eq:boundarydisc} if and only if either (1) $g$ is the constant map $g(\tau) = z_0$ for all $\tau\in \mathbb D$ or (2) for some $t$ the weak-Hele shaw domain $\Omega_t$ for $\omega$ is simply connected and the map $\tau\mapsto \tau g(\tau)$ is a Riemann-map from $\mathbb D$ to $\Omega_t$ taking $0$ to $z_0$ or (3) $g(\tau) = \tau^{-1} z'$ for some fixed $z'\in \Omega_1$. 
\end{corollary}
\begin{proof}
This is immediate from Theorem \ref{thm:holdiscs}, since Lemma \ref{lem:twist} implies that the graph of $f$ is a harmonic disc for $\tilde{\Phi}$ if and only if the graph of $g(\tau) = \tau^{-1} f(\tau)$ is a harmonic disc for $\Phi$.
\end{proof}

\begin{example}
The above may be used to produce examples of boundary conditions for the HMAE over the (punctured) disc for which the weak solution to the HMAE is regular.  For suppose $X=\mathbb P^1$ with coordinate $z\subset \mathbb C\subset \mathbb P^1$ and $\Omega_t$ for $t\in (0,1)$ is a smoothly varying family of simply connected domains with the property that $\Omega_t$ is a symmetric disc around $z=0$ with area equal to $t$ (taken with respect to the Fubini-Study form $\omega_{FS}$) for $t<\epsilon$ and $t>1-\epsilon$.    We will see in \S \ref{sec:slits} that $\{\Omega_t\}_{t\in (0,1)}$ is the weak Hele-Shaw flow with respect to some K\"ahler form $\omega_{FS} +dd^c\phi$ where $\phi\in \mathcal K(X,\omega_{FS})$.  Thus Theorem \ref{thm:holdiscs}, the weak solution to the HMAE with boundary data determined by $\phi$ will be regular (the reader will find essentially the same example in \cite{Donaldson}).
\end{example}

We next discuss an interesting link between the Riemann map, the Hele-Shaw flow and the family of forms coming from the solution $\tilde{\Phi}$ to the HMAE.  
Continue to assume $X=\mathbb P^1$, $z_0$ is the origin in the chart $\mathbb C_z\subset \mathbb P^1$, and for each $\tau\in \overline{\mathbb D}^\times$ set
$$\omega_{\tau} : = \omega + dd^c \tilde{\Phi}(\cdot, \tau).$$
Then $\omega_1 = \omega$, but in general $\omega_{\tau}$ is a semipositive $(1,1)$-current on $X$ (not necessarily smooth).  One can define the weak Hele-Shaw flow with respect to such $\omega_{\tau}$ in precisely the same way as the smooth case, and we denote the associated Hele-Shaw domains by $\Omega_{t}^{\omega_\tau}$.   For $r>0$ set $\mathbb D_r = \{ z\in \mathbb C : |z|<r\}$.

\begin{proposition}\label{prop:innerdisc}
Suppose $t$ is such that $\Omega_t^\omega \subset \mathbb C_z\subset \mathbb P^1$ is proper and simply connected and let $f_t:\mathbb D\to \Omega_t^\omega$ be a Riemann-map with $f(0)=0$.  Then for each $\tau\in \mathbb D^\times$ 
$$ \Omega_t^{\omega_{\tau}}=f_t(\mathbb D_{|\tau|}).$$
\end{proposition}

\begin{proof}
Fix $\sigma\in \mathbb D^{\times}$ and set $r: = |\sigma|$, so our aim is to show $ f_t(\mathbb D_r) = \Omega_t^{\omega_\sigma}$.   As $\tilde{\Phi}$ is invariant under $(z,\tau)\mapsto (z,e^{i\theta} \tau)$ \eqref{eq:hmaelelong4} we may as well assume $\sigma$ is real, so $\omega_{\sigma} = \omega_r$. 

For a function $F$ on $\mathbb P^1\times \overline{\mathbb D}$ and $D\subset \overline{\mathbb D}$ we write $F|_{D}$ for the restriction of $F$ to $\mathbb P^1\times D$. Then $\tilde{\Phi}|_{\overline{\mathbb D}_r}$ is the solution to the Dirichlet problem for the HMAE with boundary data $\tilde{\Phi}(\cdot,\tau)_{\tau\in \partial \mathbb D_r}$ and the requirement that $\tilde{\Phi}|_{\overline{\mathbb D}_r}$ has Lelong number 1 at the point $(z_0,0)\in \mathbb C_z\times \mathbb D_r\subset \mathbb P^1\times \mathbb D_r$.

Letting $s:=-\ln |\tau|^2$ consider again
 $$H(z,\tau):=\frac{\partial}{\partial s^+} \tilde{\Phi}(z,e^{-s/2})\text{ for }(z,\tau)\in \mathbb P^1\times \overline{\mathbb D}^\times$$ 
 which is well-defined and Lipschitz (Theorem \ref{thm:c11}).    Clearly this is compatible with restriction, i.e.\
 $$ H|_{\overline{\mathbb D}^\times_r}(z,\tau) = \frac{\partial}{\partial s^+} \tilde{\Phi}|_{\overline{\mathbb D}_r}(z,e^{-s/2}).$$

By Theorem \ref{thm:holdiscs},  $\tilde{\Phi}$ is $\pi_X^*\omega$-harmonic along the graph of $f$ and $ H(f(\tau),\tau) = t-1$.   Now $H$ is also $S^1$-invariant and so this in particular implies
$$ H(f(re^{i\theta}),r)  = H(f(re^{i\theta}), re^{i\theta}) = t-1 \text{ for all } \theta\in \mathbb R.$$
In other words the function $H(\cdot, r)$
takes the value $t-1$ on the boundary of $f(\mathbb D_r)$.    On the other hand Proposition \ref{prop:Hu_t} implies
$$ H(z,r) + 1 = \sup\{ s: z\notin \Omega_s^{\omega_r}\}$$
(we remark the proof of Proposition \ref{prop:Hu_t}  does not require smoothness or  strict positivity assumptions of $\omega_r$).  Thus $\Omega_t^{\omega_r}$ is the interior component 
of the curve $\theta\mapsto f(re^{i\theta})$ (that is, the component containing the point $z=0$), which gives $\Omega_t^{\omega_r} = f(\mathbb D_r)$ as claimed.
\end{proof}

\section{The Strong Hele-Shaw Flow}\label{sec:strong}

We turn next to the strong Hele-Shaw flow.   Although it is certainly possible to consider this on a general Riemann surface, for ease of exposition we shall consider only the case of the complex plane.  We will, however,  take the flow with respect an arbitrary area form, which generalises the classical case in which the plane is usually equipped with the standard Euclidean structure.

\subsection{Definitions}
Let $0<a<b<\infty$ and suppose $\{\Omega_t\}_{t\in (a,b)}$ is a family of \emph{smoothly} bounded domains in $\mathbb C$.    By this we mean given any $t_0\in (a,b)$ and any point $p\in \partial \Omega_{t_0}$ there exists real coordinates $x,y$ on an open set $U\subset \mathbb C$ containing $p$ such that 
$$\partial \Omega_{t_0} \cap U = \{ (x,y) : y = g_{t_0}(x) \}$$
for some smooth function $g_{t_0}$.  We also assume this family is \emph{smooth}, by which we mean one can pick $U$ so that $g_t$ is smooth in $t$ for $t$ close to $t_0$.  As a last assumption we assume also $\Omega_t$ is increasing, so $\Omega_t\subset \Omega_{t'}$ for $t<t'$.

So if $n$ denotes the outward unit normal vector field $n$ on $\partial \Omega_{t_0}$ then for $t$ close to $t_0$ we can write $\partial \Omega_t = \{ x + f(x,t) n_x : x\in \partial \Omega_{t_0}\}$ for some smooth function $f_t(x) = f(x,t)$ on $\partial \Omega_{t_0}$ that is positive for $t>t_0$ and negative for $t<t_0$.     The \emph{normal velocity} of $\partial \Omega_{t_0}$ is defined to be
$$ V_{t_0} := \frac{df_t}{dt}\big\vert_{t=0} n.$$

We will take the origin $0$ as our distinguished point, and assume  $0\in\Omega_t$ for all $t$.    For each $t$ let $$p_t(z):=-G_{\Omega_t}(z)$$ where $G_{\Omega_t}$ denotes the Green's function for $\Omega_t$ with logarithmic singularity at the origin. Thus 
$$ p_t = 0 \text{ on } \partial \Omega_t \text{ and } \Delta p_t =-\delta_0.$$
The statement that $p_t$ exists and is smooth on $\overline{\Omega}_t\setminus \{0\}$ is classical.   We also fix a smooth area form on $\mathbb C$ which we write as 
$$\frac{1}{\kappa} dA$$
where $dA=dx\wedge dy$ is the standard Lebesgue measure and $\kappa$ is a strictly positive real-valued smooth function on $\mathbb C$.  

\begin{definition}(Strong Hele-Shaw flow)
We say $\{\Omega_t\}_{t\in (a,b)}$ is the \emph{strong Hele-Shaw flow} if
\begin{equation}\label{eq:HSclassical:again} 
 V_t = -\kappa \nabla p_t \text{ on } \partial \Omega_t \text{ for } t\in (a,b)
\end{equation}
where $V_t$ is the normal velocity of $\partial \Omega_t$.   When necessary to emphasise the dependence on the area form we refer to this as the strong Hele-Shaw flow with respect to the area form $\frac{1}{\kappa}dA$ (or with respect to $\kappa$).
\end{definition}

The above has the following physical interpretation.  Consider two parallel plates infinite in all directions separated by a small gap.  Suppose between these two plates is some porous medium with varying permeability, and a fluid is injected into the gap through a fixed point in one of the plates at a constant rate.      As the gap between the plates is small, this is essentially a two-dimensional flow that is modelled by the region $\Omega_t$ that the fluid occupies at time $t$.   We may as well assume the fluid is injected at the origin.  Then the permeability of the medium is encoded by a function $\kappa:\mathbb C\to \mathbb R_+$, so the fluid moves more freely in the areas of the plane in which $\kappa$ is relatively big.   The function $p_t$ models the pressure of the system, and we make some physical assumptions, namely the fluid is incompressible (meaning $p_t$ is harmonic away from the origin) and the medium itself does not exert any pressure on the system (meaning that $p_t$ is constant on the boundary, so after subtracting a constant we may as well take to be zero).    The equation of motion \eqref{eq:HSclassical:again} for the strong Hele-Shaw flow is then a case of Darcy's law which describes the flow of a fluid through a porous medium.
 
\subsection{Strong implies weak}
Our next goal is to prove the strong Hele-Shaw flow is also a weak one.    To do so, we start with a famous calculation due to Richardson \cite{Richardson}.

\begin{lemma}\label{lem:richardson}
Suppose $\{\Omega_t\}_{t\in (a,b)}$ is a  strictly increasing smooth family of simply connected domains in $\mathbb C$ containing the origin that satisfies
\begin{equation}\label{eq:HSclassical:repeat} 
 V_t = -\kappa \nabla p_t \text{ on } \partial \Omega_t 
\end{equation}
as in \eqref{eq:HSclassical:again}.    Then for any integrable subharmonic function $h$ on $\Omega_t,$ and $t_0<t$ 
$$  \int_{\Omega_t\setminus \Omega_{t_0}} h \frac{dA}{\kappa}  \ge (t-t_0)  h(0).$$ 
\end{lemma}
\begin{proof}
We compute using the Reynolds transport theorem,
\begin{align}\label{eq:richardson}
  \frac{d}{dt} \int_{\Omega_t}h \frac{1}{\kappa} dA &= \int_{\partial \Omega_t} h \frac{V_t}{\kappa} ds = -\int_{\partial \Omega_t} h \frac{\partial p_t}{\partial n} ds \\
&= \int_{\Omega_t} \left(p_t \Delta h - h \Delta p_t\right) dA - \int_{\partial \Omega_t} p_t \frac{\partial h}{\partial n} ds\ge h(0)\nonumber
\end{align}
since $\Delta h\ge 0$ and $p_t=0$ on $\partial \Omega_t$ and $\Delta p_t = -\delta_0$.  
\end{proof}

\begin{corollary}\label{cor:moment}
With the assumption of the above lemma, suppose  $a=0$ and $\Omega_t$ tends to $\{0\}$ as $t\to 0$ (i.e.\  given any neighbourhood $U$ of the origin $\Omega_t\subset U$ for $t$ sufficiently small).  Then for any integrable subharmonic function $h$ on $\Omega_t$, 
\begin{equation} \int_{\Omega_t} h \frac{dA}{\kappa}  \ge t  h(0)\label{eq:momentineq}\end{equation}
and  equality holds if $h$ is harmonic.     In particular 
\begin{align}
  \int_{\Omega_t} \ln |z-\zeta|^2 \frac{dA_\zeta}{\kappa(\zeta)} &= t  \ln |z|^2 \text{ for } z\notin \Omega_t, \label{eq:momentineq2}\\
  \int_{\Omega_t} \ln |z-\zeta|^2 \frac{dA_\zeta}{\kappa(\zeta)} &> t  \ln |z|^2 \text{ for } z\in \Omega_t. \label{eq:momentineq3}
  \end{align}
\end{corollary}
\begin{proof}
Taking the limit as $t_0\to 0$ in the above Lemma gives  \eqref{eq:momentineq}  The statement about harmonic functions follows as if $h$ is harmonic then $h$ and $-h$ are subharmonic.   Equation \eqref{eq:momentineq2} follows as if $z\notin \Omega_t$ then $h(\zeta):=\ln |z-\zeta|^2$ is harmonic for $\zeta \in \Omega_t$.  If $z\in \Omega_t$ then $\Delta \ln |z-\zeta|^2 = 2 \delta_{z}$, so in Richardson's calculation \eqref{eq:richardson}
$$ \frac{d}{dt} \int_{\Omega_{t}}h \frac{1}{\kappa} dA \ge \int_{\Omega_{t}} 2 p_{t} \delta_{z} + h(0) > h(0)$$
from which one deduces the strict inequality in \eqref{eq:momentineq3}.
\end{proof}

\begin{proposition}[Gustafsson]\label{prop:gust}
Suppose  $\{\Omega_t\}_{t\in (0,b)}$ is a smooth family of strictly increasing simply connected domains that is the strong Hele-Shaw flow with respect to $\kappa$, and assume $\{\Omega_t\}_{t\in (0,b)}$ tends to $\{0\}$ as $t\to 0$.  Then the weak Hele-Shaw envelope with respect to the K\"ahler form
 $$\omega : = \frac{1}{\kappa} dA$$ is given by
$$ {\psi_t}(z) = - \int_{\Omega_t}  \log |z-\zeta|^2 \frac{dA_\zeta}{\kappa(\zeta)} + t \ln |z|^2,$$
and
$\{\Omega_t\}_{t\in (0,b)}$ is the weak Hele-Shaw flow with respect to $\omega$.
\end{proposition}
\begin{proof}
For the proof let
$$ \tilde{\psi_t}(z) := - \int_{\Omega_t}  \log |z-\zeta|^2 \frac{dA_\zeta}{\kappa(\zeta)} + t \ln |z|^2$$
and write $\Omega_t^w:= \{ z\in X : \psi_t(z)< 0\}$ for the weak Hele-Shaw flow with respect to $\omega$. 
So the goal is to prove $\tilde{\psi}_t = \psi_t$ and $\Omega_t^w= \Omega_t$

For large $R$ let $B_R = \{ |z|<R\}$ and set
$$ \phi(z) = \int_{B_R} \log |z-\zeta|^2 \frac{dA_\zeta}{\kappa(\zeta)} \text{ for } z\in \mathbb C.$$
Then on $B_R$, $dd^c\phi=\omega$ and 
$$\omega_{\tilde{\psi}_t} = dd^c (\phi + \tilde{\psi}_t) = dd^c \int_{B_R\setminus \Omega_t} \ln |z-\zeta|^2 \frac{dA_{\zeta}}{\kappa(\zeta)}\ge 0.$$ 
As $R$ can be arbitrarily large this implies $\tilde{\psi}_t\in \Sh(\mathbb C,\omega)$.  Clearly $\nu_0(\tilde{\psi}_t) = t$ and (\ref{eq:momentineq2},\ref{eq:momentineq3}) imply $\tilde{\psi}_t\le 0$ with equality on $\Omega_t^c$.  Thus $\tilde{\psi}_t$ is a candidate for the envelope defining the Hele-Shaw envelope, so $\tilde{\psi}_t\le \psi_t$ giving  $\Omega_t^w\subset \Omega_t$.  Now both $\psi_t$ and $\tilde{\psi}_t$ have Lelong number precisely $t$ at the origin, the maximum principle implies $\psi_t\le \tilde{\psi}_t$ over $\Omega_t$, and so $\psi_t = \tilde{\psi}_t$ everywhere, and $\Omega_t\subset \Omega_t^w$ follows from   \eqref{eq:momentineq3}.

%Therefore, if $z\in \Omega_t^w$ then $\tilde{\psi}(z)\le \psi_t(z)<0$ so $z\in \Omega_t$, i.e. .    So actually
%$$ \tilde{\psi}_t =\psi_t= 0  \text{ on } \Omega_t^c.$$
  \end{proof}

\subsection{Weak and Smooth implies Strong}

We now show if the weak Hele-Shaw flow is smooth and smoothly varying, then it is in fact the strong Hele-Shaw flow.

\begin{lemma}\label{lem:momentimpliesstrong}
Suppose $\{\Omega_t\}_{t\in (0,t_0)}$ is a smoothly varying family of bounded increasing domains, such that for any function $h$ that is harmonic on $\overline{\Omega}_t,$  
\begin{equation}  \int_{\Omega_t} h \frac{dA}{\kappa}  = t  h(0) \label{eq:harmonic}\end{equation}
Then $\{\Omega_t\}_{t\in (0,t_0)}$ is the strong Hele-Shaw flow with respect to $\kappa$.
\end{lemma}
\begin{proof}
This is Richardson's calculation backwards.    Let $h$ be as in the statement.  Then using the hypothesis \eqref{eq:harmonic}
$$  h(0) = \frac{d}{dt} \int_{\Omega_t}h \frac{dA}{\kappa}  = \int_{\partial \Omega_t} h \frac{V_t}{\kappa} ds.$$
On the other hand,  just as in \eqref{eq:richardson}
$$  -\int_{\partial \Omega_t} h \frac{\partial p_t}{\partial n} ds = \int_{\Omega_t} \left(p_t \Delta h - h \Delta p_t\right) dA - \int_{\partial \Omega_t} p_t \frac{\partial h}{\partial n} ds= h(0).$$
Therefore
$$ \int_{\partial \Omega_t} h\left( \frac{V_t}{\kappa}-\frac{\partial p_t}{\partial n} \right) ds=0$$
and since this holds for all such harmonic functions we must have
$$\frac{V_t}{\kappa}=\frac{\partial p_t}{\partial n} \text{ on } \partial \Omega_t$$
making $\{\Omega_t\}_{t\in (0,t_0)}$ the strong Hele-Shaw flow
\end{proof}

\begin{corollary}
Suppose for some $t_0$ the weak Hele-Shaw domains $\{\Omega_t\}_{t\in (0,t_0)}$ taken respect to 
$$\omega: = \frac{1}{\kappa} dA$$
are bounded and smooth (i.e.\ each $\Omega_t$ is smoothly bounded and varies smoothly and each $\Omega_t$ is bounded for $t<t_0$).    Then $\{\Omega_t\}_{t\in (0,t_0)}$ is the strong Hele-Shaw flow with respect to $\kappa$.
\end{corollary}
\begin{proof}
Let $h$ be harmonic on $\overline{\Omega}_t$.    By Proposition \ref{prop:basicplane}(4),  $\omega_{\psi_t} = (1-\chi_{\Omega_t}) \omega + t\delta_0$, where $\psi_t$ is the Hele-Shaw envelope, giving
$$\int_{\Omega_t} h \frac{dA}{\kappa} = \int_{\Omega_t} h \omega = -\int_{\Omega_t} h dd^c\psi_t + th(0) =th(0)$$
where the last equality uses Greens formula applied to a smooth domain containing $\Omega_t$ on which $h$ is harmonic.  Thus the result follows from Lemma \ref{lem:momentimpliesstrong}.
\end{proof}

\subsection{Bibliographical remarks}

The weak and strong point of view for the Hele-Shaw flow is a theme in the work of Gustafsson (e.g. \cite{Gustafsson4,Gustafsson3,Gustafsson2}), and the reader interested in more is referred again to \cite{Gustafssonbook}.    Classically this flow is considered with respect to the standard area form (Lebesgue measure), with a given initial domain $\Omega_0$.    The first problem then becomes proving short time existence of the Hele-Shaw flow, a result that goes back to Kufarev--Vinogradov \cite{Vinogradov} who prove that for a simply connected initial domain with real analytic boundary the strong Hele-Shaw flow (taken with respect to the standard Lebesgue measure) exists for some interval both forwards and backwards in time.  This has then been reproved in various forms in \cite{Gustafsson3,Lin,Reissig,Tian}.

% which under various conditions on $\Omega_0$ can be found in various works going back to Kufarev--Vinogradov \cite{Vinogradov} (see also \cite{Gustafsson3,Lin,Reissig, Tian}).

It is not really interesting to consider the case of empty initial condition in the classical case, as then the flow consists simply of concentric discs centered at the origin.  However, if one allows a general area form, then the problem of short-term existence of the Hele-Shaw flow with empty initial condition is non-trivial.    Under the assumption that the area form is analytic and hyperbolic this short term existence is due to Hedenmalm-Shimorin \cite{Hedenmalm}, and when the area form is merely smooth by the authors  \cite{RWDisc}.  That is, given an arbitrary smooth area form, there exists an $\epsilon>0$ such that the strong Hele-Shaw flow exists for $0<t<\epsilon$ and tends to $\{0\}$ as $t$ tends to $0$.  Moreover, as long as $\epsilon$ is sufficiently small, each $\Omega_t$ is smoothly bounded and simply connected.  The proof that we give, and the only one known at present, comes about through the connection between the Hele-Shaw flow and the Monge-Amp\`ere foliation.  First, using a form of Schwarz function, we interpret a simply connected Hele-Shaw domain as a holomorphic disc with boundary in a totally real submanifold (just as in Donaldson's LS-submanifolds).  This converts the short term existence problem of the Hele-Shaw flow to a problem about deforming such holomorphic discs, which is a well-known elliptic problem.  The reader is referred to \cite{RWDisc} for details.

Richardson's calculation represents an important viewpoint of the Hele-Shaw flow (see \cite{Gustafssonexponential} for a survey).     Putting $h(z) = z^k$ for $k\in \mathbb N_{\ge 1}$, equation \eqref{eq:momentineq} says that for the strong Hele-Shaw flow the ``complex moments"
$$ M_k(t) : = \int_{\Omega_t} z^k \frac{dA}{\kappa}$$
are constant with respect to $t$.   This illustrates the fundamental nature of the Hele-Shaw flow. For  assuming that $\kappa$ is analytic and simply connected domain $\Omega_{t_0}$ with analytic boundary, the set $\{M_k(t)\}$ form local coordinates for the set of analytic perturbations of $\Omega_{t_0}$ (that is, any nearby domain with analytic boundary is uniquely specified by its complex moments).   So any such flow starting at $\Omega_{t_0}$ can, in principle, be described by its change in complex moments.  Thus the Hele-Shaw flow is the simplest among all possible flows, and  with this viewpoint it is not surprising that it appears in so many parts of pure and applied mathematics.

\section{Examples}\label{sec:examples}

We work throughout with $X=\mathbb P^1$ with K\"ahler form $\omega$ normalised so $\int_{\mathbb P^1}\omega =1$.   We consider $\mathbb P^1$ covered by two copies of $\mathbb C$ with coordinates $z$ and $w=1/z$ respectively (we denote these two charts by $\mathbb C_z$ and $\mathbb C_w$) and let $z_0$ be the point $z=0$.   In each case we will deduce information about the solution $\tilde{\Phi}$ to the HMAE over the punctured disc.  The interested reader will easily be able to translate these to similar statements for the HMAE over the disc using Lemma \ref{lem:twist}.

\subsection{Flows developing self-tangency}\label{sec:slits}

\begin{definition}\label{def:tangency} 
 We say the Hele-Shaw for  \emph{develops self-tangency} at a point $p\in \mathbb C_z\subset \mathbb P^1$ if there exists a $t_0>0$ such that 
 \begin{enumerate}
 \item $\Omega_t$ is smoothly bounded, simply connected and varies smoothly for $t<t_0$ and 
 \item  $\Omega_{t_0}$ is a simply connected in $\mathbb C_z$ and $\partial \Omega_{t_0}$ is the image of a smooth locally embedded curve intersecting itself tangentially precisely at the point $p$ (see Figure \ref{fig1}).
 \end{enumerate}
\end{definition}

\begin{figure}[htb]
\centering
\scalebox{1}{\input{tangency0.pspdftex}}
	\caption{The Hele-Shaw flow developing self-tangency I}
	\label{fig1}
\end{figure}
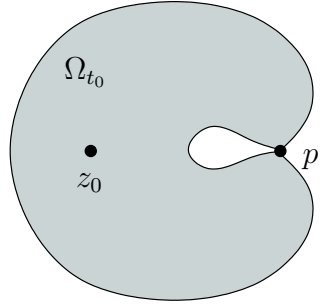

\begin{theorem}[Ross-Witt Nystr\"om]\label{thm:notc2}
Suppose the Hele-Shaw flow for $\omega$ develops self-tangency at $p$.    Then the weak solution $\tilde{\Phi}$  to the HMAE is not twice differentiable at the points $(p,\tau)$ for $|\tau|=1$. 
\end{theorem}

Rather than giving a full proof we illustrate this with an instructive example.  Say $(x,y)$ are smooth coordinates centered at $p$, and that near $p$ 
$$\Omega_t = \{ y < -x^2 - (t_0-t) \} \cup \{ y > x^2 + (t_0-t)\}\quad\text{ for } t_0<t.$$
%\{ (x,y): -x^2 -(t_0-t) \le y \le x^2 + (t_0-t)\}\quad \text{ for } t_0<t.$$
%the boundary of the Hele-Shaw flow is given by
%$$\partial \Omega_{t} = \{ y = x^2 + (t_0-t) \} %\cup \{ y = -x^2 - (t_0-t)\} \text{ for }t\le %t_0.$$
Set 
$$h(x,y):= H((x,y),1)$$ 
where as usual
$$H(z,\tau):=\frac{\partial}{\partial s^+} \tilde{\Phi}(z,e^{-s/2}),$$
%Notice if $|y|$ is sufficiently small then $(0,\pm y)$ lies in $\Omega_{t}$ for $t$ sufficiently close to $t_0$. 
and recall by Proposition \ref{prop:Hu_t} 
$$H(z,1)+1= \sup\{ t: z\notin \Omega_t\}.$$
Thus for $|y|$ sufficiently small
$$ h(0,y) = \left \{ \begin{array}{cc} t_0-y-1& y>0 \\ t_0+y-1 & y<0\end{array}\right.$$
and from this it is clear $\frac{\partial h}{\partial y}$ does not exist at the origin, and so $\tilde{\Phi}$ is not twice differentiable at $(p,1)$.\medskip

Of course, for this idea have any use, we need to be able to ensure the Hele-Shaw can develop self-tangency.    To do so we start by showing essentially any reasonable family of simply connected domains is the Hele-Shaw flow with respect to \emph{some} smooth area form $\frac{1}{\kappa} dA$.    Assume for $t\in (a,b)$ that $\Omega_t$ is smoothly bounded, smoothly varying simply connected and strictly increasing and each contains the origin.   Take $p_t$ to be defined by
$$ p_t = 0 \text{ on } \partial \Omega_t \text{ and } \Delta p_t =-\delta_0.$$
As already  mentioned, the fact  $p_t$ exists and is smooth on $\overline{\Omega}_t\setminus \{0\}$ is classical.   What is also true is $p_t$ varies smoothly with $t$ (it seems to the authors that all the known proofs of the existence of $p_t$ actually prove this stronger statement, see for instance \cite[Appendix A]{RWApplications}).    

Then (as observed by Berndtsson) one can reverse-engineer the defining equation for the Hele-Shaw flow to define a smooth function $\kappa$ by requiring 
\begin{equation}
 V_t = -\kappa \nabla p_t \text{ on } \partial \Omega_t\text{ for } t\in (a,b).\label{eq:HSclassical}
\end{equation}
Since $\{\Omega_t\}$ is assumed to be strictly increasing,  $V_t$ is non-vanishing so $\kappa$ is a well-defined strictly positive smooth function on some subset of $\mathbb C$.  If we further assume $a=0$ and for $t$ sufficiently small $\Omega_t$ is just a disc centred at the origin with Lebesgue area $t$, then $\kappa$ is constant on $\partial \Omega_t$ for $t$ sufficiently small, and thus extends to a smooth function across the origin.  So, by construction, $\{\Omega_t\}_{t\in (0,b)}$ is the strong Hele-Shaw flow with respect to $\frac{1}{\kappa} dA$.   So far we have defined a smooth $\kappa$ on $\Omega_{b}$.  Assuming that $\kappa$ extends to a smooth function on $\overline{\Omega}_b$, we may then extend it to a smooth function on all of $\mathbb P^1$, giving an area form whose Hele-Shaw flow agrees with  $\{\Omega_t\}$ for $t<b$.

We can now sketch how to use this to produce an area form whose Hele-Shaw flow develops self-tangency (see the right hand side Figure \ref{fig2} and observe that in this figure have moved our distinguished point $z_0$ to be the point $-1$).     Fix $t_0\in (0,1)$ and let $\tilde{\Omega}_{t_0}$ be as in the figure.  We assume $\tilde{\Omega}_{t_0}$ has analytic boundary, and is symmetric under $x+iy\mapsto -x+iy$.  Let $\tilde{z_0}:=-i\in \tilde{\Omega}_{t_0}$ so $\tilde{z}_0^2= z_0=-1$.  

\begin{figure}[htb]
\centering
\scalebox{1}{\input{tangency.pspdftex}}
	\caption{The Hele-Shaw flow developing self-tangency II}
	\label{fig2}
\end{figure}
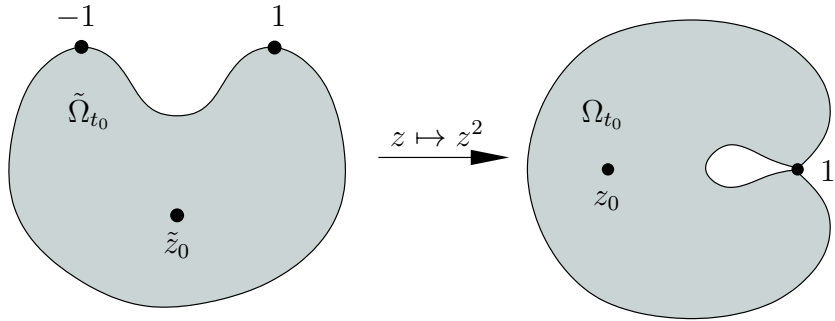

Consider the form $\tilde{\omega} :=\frac{i}{\pi} |z|^2 dz\wedge d\overline{z}$ which is real analytic and strictly  positive away from $z=0$.     Then from short time existence of the strong Hele-Shaw flow with analytic initial conditions \cite[Theorem 6.2]{Hedenmalm}, there is a $\delta>0$ such that the strong Hele-Shaw flow  with initial condition $\tilde{\Omega}_{t_0}$ with injection point $\tilde{z}_0$ taken respect to $\tilde{\omega}$ exists for $t\in [t_0-\delta,t_0 + \delta]$.   Next observe $\tilde{\omega}$ is the pullback of the standard K\"ahler form $\omega: = dd^c |z|^2$ on $\mathbb C$ by the map $f(z) = z^2$.    Set  $\Omega_t: = f(\tilde{\Omega}_t)$ for $t\in [t_0-\delta,t_0]$, so by construction $\Omega_{t_0}$ is self-tangent at the point $p=1$.  It is not hard to show that $\{\Omega_t\}_{t\in [t_0-\delta,t_0]}$ is the strong Hele-Shaw flow with respect to $\omega$.    We then complete this to a flow that tends to the point $z_0$ as $t$ tends to zero by taking $\Omega_{t_0-\delta}$ and shrinking it smoothly towards  $z_0$.  Our previous discussions show that it is possible to do so in such a way to obtain a K\"ahler form on $\mathbb P^1$ whose Hele-Shaw flow agrees with $\Omega_t$ for $t\le t_0$, and thus develops self-tangency as desired.

\subsection{Multiply-connected flows}\label{sec:multiply}

Using what has already been said, it is not hard to show there are K\"ahler forms on $\mathbb P^1$ whose corresponding Hele-Shaw flow ceases to be simply connected at some point in time.  One way to arrange this is to use the flow from the previous section that develops self-tangency at a point $p$ at time $t_0$, so for a short time after $t_0$ the domain $\Omega_t$ will not be simply connected.  Another way to produce such an example is to start with a K\"ahler form that puts almost all of its mass on a given annulus $A\subset \mathbb P^1$ containing $z_0$.  Physically this means the Hele-Shaw flow is modelling a fluid moving through a medium that has very high permeability on $A$, and low permeability outside of $A$.   Intuitively one expects that the Hele-Shaw domains will rapidly wrap around within $A$ before it has a chance to completely cover the bounded domain in the complement of $A$, thus giving a flow that at some point becomes non simply connected.  This idea can be made rigorous, and we refer the reader to \cite[Proposition 1.4]{RWHarmonicDiscs} for details.

\begin{theorem}[Ross-Witt Nystr\"om]\label{thm:foliation}
Suppose $\omega$ is a K\"ahler form on $\mathbb P^1$ and there exist two times $t_0<t_1$ such that the weak Hele-Shaw domains $\Omega_{t}$ with respect to $\omega$ is not simply connected for any $t\in (t_1,t_2)$.   Then there exists an open set $U\subset\mathbb P^1\times \overline{\mathbb D}$  intersecting $\mathbb P^1\times \partial{\mathbb D}$ non-trivial that does not meet any proper harmonic disc of $\tilde{\Phi}$.
\end{theorem}

\begin{proof}
Theorem \ref{thm:holdiscs} lists all the harmonic discs, and also says the function $H$ is constant on any harmonic disc of $\tilde{\Phi}$.  From this one sees that no such disc can intersect the open set $$U:=\{(z,\tau): t_1-1<H(z,\tau)<t_2-1, |\tau|>0\}.$$ Since $H(z,1)$ is continuous, and attains both values $-1$ and $0$ somewhere on $X$, it follows from continuity that $U\cap(\mathbb{P}^1\times \partial \mathbb{D})$ is non-empty.  
\end{proof}

The point of this statement is it implies the solution $\tilde{\Phi}$ to the HMAE is far away from being regular, since the existence of $U$ obstructs the possibility of a foliation of $\mathbb P^1\times \overline{\mathbb D}$ by proper harmonic discs. It is interesting to compare with the example of Gamelin and Sibony, Example \ref{ex:gamelin}. There the set of proper harmonic discs did also not foliate the whole domain (which in this case was the unit ball in $\mathbb{C}^2$) but the boundaries of those discs did foliate the boundary of the domain. In our example we see even this is not the case.

\subsection{Flows with simply connected final domains}\label{sec:ex3}

Our third example concerns Hele-Shaw flows on $\mathbb P^1$ whose final domain is biholomorphic to the disc.      Suppose $\gamma$ is a non-trivial curve in $\mathbb C_w$ through the point $w=0$ (i.e.\ the point $z=\infty$).  That is, $\gamma$ is image of a smooth function $[0,1]\to \mathbb C_w$ that does not intersect itself and passes through $w=0$.   

\begin{theorem}[Ross-Witt Nystr\"om] Suppose the final Hele-Shaw domain of $\omega$ satisfies
$$\Omega_1 = \mathbb P^1\setminus \gamma.$$
There there is an open subset $S\subset \mathbb P^1\times \mathbb D$ such that the solution $\tilde{\Phi}$ to the HMAE satisfies
$$ \pi_{\mathbb P^1} \omega + dd^c \tilde{\Phi} =0 \text{ on } S.$$
\end{theorem}

Said another way, we already know the rank of the form $\pi_{\mathbb P^1} \omega + dd^c \tilde{\Phi}$ can be at most 1, since $(\pi_{\mathbb P^1} \omega + dd^c \tilde{\Phi})^2 =0$.  Thus the above gives an open subset $S$ on which $\pi_{\mathbb P^1} \omega + dd^c \tilde{\Phi}$ fails to have maximal rank.

\begin{proof}

We shall prove the slightly weaker statement that for each $\tau\in \mathbb D^\times$ the current $\omega + dd^c\tilde{\Phi}(\cdot,\tau)$  vanishes on some non-empty open subset of $\mathbb P^1$ (and the reader is referred to \cite{RossNystrommaximal} for the proof of the full statement).     As
 $$\Omega_1 =\mathbb P^1\setminus \gamma,$$
 and $\gamma$ passes through the point $w=0$, we see $\Omega_1$ is a simply connected proper subset of $\mathbb C_z$.  Consider the Riemann map $f:\mathbb D\to \Omega_1$ with $f(0)= 0$.  Then by Proposition \ref{prop:innerdisc}
$$ A_\tau: = f(\mathbb D_{|\tau|}) = \Omega_1^{\omega_\tau} \text{ for } \tau\in \mathbb D^{\times}.$$
In particular, $A_{\tau}$ is a proper subset of $\mathbb C_{z}$ whose complement has non-empty interior if $|\tau|<1$.  

On the other hand, for all $t\in [0,1]$
$$\int_{\Omega_t^{\omega_\tau}} \omega_\tau = t.$$
(we saw this statement Corollary \ref{cor:indofphi} under the assumption that $\omega_{\tau}$ is a K\"ahler form, and this more general statement can be deduced using \cite[Remark 1.19, Corollary 2.5]{BermanDemailly}).    Therefore
$$\int_{A_\tau} \omega_\tau = \int_{\Omega_1^{\omega_\tau}}  \omega_\tau = t.$$
But our normalisation is that $\int_{\mathbb P^1} \omega_{\tau}= \int_{\mathbb P^1} \omega =1$, and so $\omega_\tau$ gives zero measure to the complement of $A_{\tau}$, which is precisely what we were aiming to prove.
\end{proof}

It is not hard to construct a specific example of a K\"ahler metric on $\mathbb P^1$ for which $\Omega_1 = \mathbb P^1\setminus \gamma$ for some such arc $\gamma$.  To do so, let $\omega_{FS}$ be the Fubini-Study form, so $\omega = \ln (1+|w|^2)$ on $\mathbb C_w$.  We claim there is a $\phi\in \mathcal{C}^{\infty}(\mathbb P^1)$ such that $\omega:=\omega_{FS} + dd^c\phi>0$ and $\phi\ge - \ln (1+|w|^2)$ with equality precisely on $\gamma$.   One can then deduce easily that $\Omega_1 = \{ z : \phi(w) > - \ln (1+|w|^2)\} = \mathbb P^1\setminus \gamma$. 

 To produce such a $\phi$, assume for simplicity that $\gamma$ is the interval $[-1,1]\subset \mathbb R\subset \mathbb C_w$ and let $\alpha:\mathbb R\to \mathbb R$ be a non-negative smooth non-decreasing convex function with $\alpha(t) =0$ for $t\le 1$ and $\alpha(t)>0$ for $t>1$.  Then
$$u(w): = \alpha(|w|^2)  + \operatorname{Im}(w)^2$$
is a smooth strictly subharmonic function on $\mathbb C_w$ that vanishes precisely on $\gamma$.   Using a regularised version of the maximum function, one can adjust the function $\epsilon u-\ln (1+|w|^2)$  for some small constant $\epsilon>0$ to have the correct behaviour far away from $\gamma$ to ensure $\phi$ extends to a smooth function over $\mathbb P^1$ and $\omega_{FS}+dd^c\phi>0$.   The reader will find full details in \cite[Section 5.4]{RossNystrommaximal}.

\subsection{Hele-Shaw flow with acute corner points}\label{sec:acute}

Our final example exploits work of Sakai concerning the behaviour of the Hele-Shaw flow with corner points.  A point $c$ lying on the boundary $\partial \Omega_t$ of the weak Hele-Shaw domain for $t$ in some non-trivial interval is referred to as a \emph{stationary point}.  Sakai proves in \cite{Sakai5} (see also \cite[Theorem 6.2]{Sakai}) that if $\partial \Omega_{0}$ contains a corner point $c$ with angle strictly less than $\pi/2$ then $c$ is a stationary point for the weak Hele-Shaw flow starting at $\Omega_{0}$ (this is to be taken as holding in the plane with its the Euclidean structure).  

Suppose that $\Omega_{0}\subset \mathbb C_z\subset \mathbb P^1$ is such a domain and set
$$ \omega: = \frac{i}{2}(1- \chi_{\Omega_{0}}) dz \wedge d\overline{z}$$
on a large ball containing $\Omega_{0}$.  We then extend $\omega$ to a smooth K\"ahler form outside of this ball to all of $\mathbb P^1$.  Observe that $\omega$ is absolutely continuous and semipositive, but of course not smooth.    Looking back at the proofs of the Duality Theorem and it implications for the movement of the boundary of the weak Hele-Shaw flow (Corollary \ref{cor:HSmovement}) it is clear that they still hold for such $\omega$.

\begin{proposition} With background form $\omega$, the weak solution $\tilde{\Phi}$ to the HMAE is not in $\mathcal{C}^1(\mathbb P^1\times \overline{\mathbb D})$.
\end{proposition}
\begin{proof}
Essentially by definition, $\Omega_{0}$ is the weak Hele-Shaw domain at time $t=0$ with respect to $\omega$.  By the result of Sakai, the corner point of $\Omega_{0}$ is stationary, and thus by Corollary \ref{cor:HSmovement}(1) the function $H$ is not continuous, which means $\tilde{\Phi}$ is not $\mathcal{C}^1$.
\end{proof}

The implications of this can be expressed in terms of potentials.  If $\omega_{FS}$ denotes the Fubini-Study form, then (after possibly scaling $\omega$) we can write
$$ \omega = \omega_{FS} + dd^c\phi$$
for some potential $\phi$.  As $\omega$ is absolutely continuous $\phi$ has bounded Laplacian, and thus lies in $\mathcal{C}^{1,\alpha}$ for all $\alpha<1$.  On the other hand combining the previous Proposition with Lemma \ref{lem:twist}, the weak solution 
$$ \Phi: = \sup\{ \Psi\in \Psh(\mathbb P^1\times \overline{\mathbb D},\pi_{\mathbb P^1}^*\omega_{FS})  : \Psi(z,\tau)\le \phi(\tau z, \tau) \text{ for } |\tau|=1\}$$
to the HMAE is not even in the class $\mathcal{C}^1$.

\subsection{Final Bibliographical Remarks}

The final example is new, but the first three are taken from \cite{RWApplications}, \cite{RWHarmonicDiscs} and \cite{RossNystrommaximal} respectively, and the reader will find stronger statements in these cited papers.  For instance in  \cite{RWApplications}  one can find an area form whose Hele-Shaw flow develops self-tangency along any given finite collection of smooth points and non-selfintersecting curve segments.  Thus it is possible to find Dirichlet data for an HMAE that is not twice differentiable at such a prescribed set of points.   And in \cite{RWHarmonicDiscs} it is shown that the phenomena of having (smooth) Dirichlet data for the HMAE for which there is an open set not meeting any harmonic disc can be made to persist under small deformations of the data.
\bibliographystyle{amsplain}

\section*{Acknowledgements}
\noindent We wish to thank Valentino Tosatti for conversations relating to this survey, as well as the referee for helpful comments and references.

\end{document}

%% file: tangency0.pspdftex
\begin{picture}(0,0)%
\includegraphics{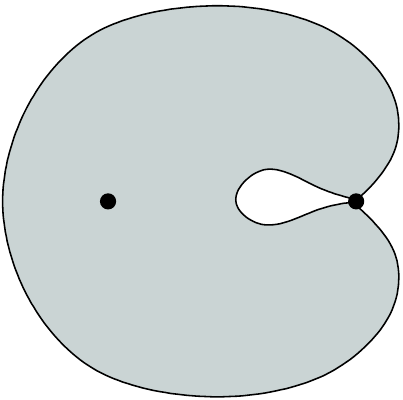}%
\end{picture}%
\setlength{\unitlength}{3947sp}%
\begingroup\makeatletter\ifx\SetFigFont\undefined%
\gdef\SetFigFont#1#2#3#4#5{%
  \reset@font\fontsize{#1}{#2pt}%
  \fontfamily{#3}\fontseries{#4}\fontshape{#5}%
  \selectfont}%
\fi\endgroup%
\begin{picture}(1926,1901)(5950,-3636)
\put(7801,-2761){\makebox(0,0)[lb]{\smash{{\SetFigFont{12}{14.4}{\rmdefault}{\mddefault}{\updefault}{\color[rgb]{0,0,0}$p$}%
}}}}
\put(6301,-2236){\makebox(0,0)[lb]{\smash{{\SetFigFont{12}{14.4}{\rmdefault}{\mddefault}{\updefault}{\color[rgb]{0,0,0}$\Omega_{t_0}$}%
}}}}
\put(6376,-2911){\makebox(0,0)[lb]{\smash{{\SetFigFont{12}{14.4}{\rmdefault}{\mddefault}{\updefault}{\color[rgb]{0,0,0}$z_0$}%
}}}}
\end{picture}%

%% file: tangency.pspdftex
\begin{picture}(0,0)%
\includegraphics{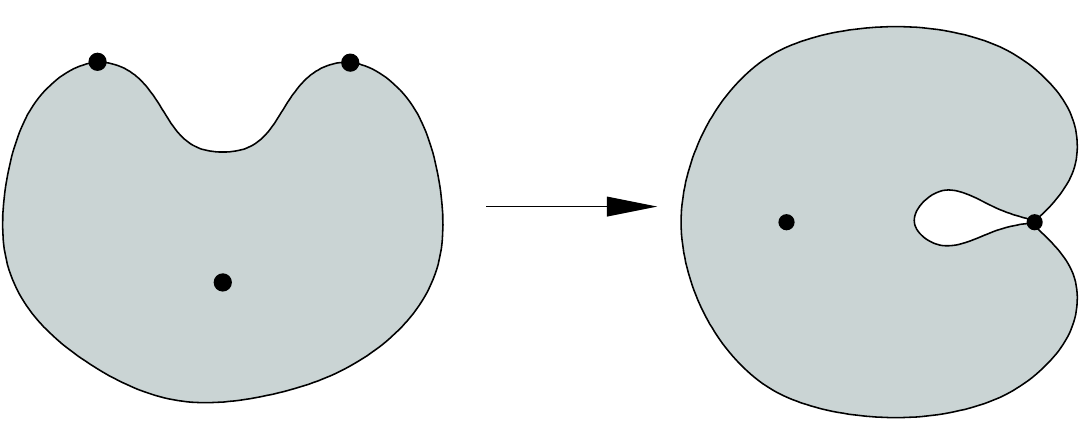}%
\end{picture}%
\setlength{\unitlength}{3947sp}%
\begingroup\makeatletter\ifx\SetFigFont\undefined%
\gdef\SetFigFont#1#2#3#4#5{%
  \reset@font\fontsize{#1}{#2pt}%
  \fontfamily{#3}\fontseries{#4}\fontshape{#5}%
  \selectfont}%
\fi\endgroup%
\begin{picture}(5183,2009)(2693,-3636)
\put(7801,-2761){\makebox(0,0)[lb]{\smash{{\SetFigFont{12}{14.4}{\rmdefault}{\mddefault}{\updefault}{\color[rgb]{0,0,0}$1$}%
}}}}
\put(3076,-2386){\makebox(0,0)[lb]{\smash{{\SetFigFont{12}{14.4}{\rmdefault}{\mddefault}{\updefault}{\color[rgb]{0,0,0}$\tilde{\Omega}_{t_0}$}%
}}}}
\put(6301,-2386){\makebox(0,0)[lb]{\smash{{\SetFigFont{12}{14.4}{\rmdefault}{\mddefault}{\updefault}{\color[rgb]{0,0,0}$\Omega_{t_0}$}%
}}}}
\put(4351,-1786){\makebox(0,0)[lb]{\smash{{\SetFigFont{12}{14.4}{\rmdefault}{\mddefault}{\updefault}{\color[rgb]{0,0,0}$1$}%
}}}}
\put(3001,-1786){\makebox(0,0)[lb]{\smash{{\SetFigFont{12}{14.4}{\rmdefault}{\mddefault}{\updefault}{\color[rgb]{0,0,0}$-1$}%
}}}}
\put(3676,-3211){\makebox(0,0)[lb]{\smash{{\SetFigFont{12}{14.4}{\rmdefault}{\mddefault}{\updefault}{\color[rgb]{0,0,0}$\tilde{z}_0$}%
}}}}
\put(5101,-2536){\makebox(0,0)[lb]{\smash{{\SetFigFont{12}{14.4}{\rmdefault}{\mddefault}{\updefault}{\color[rgb]{0,0,0}$z\mapsto z^2$}%
}}}}
\put(6376,-2911){\makebox(0,0)[lb]{\smash{{\SetFigFont{12}{14.4}{\rmdefault}{\mddefault}{\updefault}{\color[rgb]{0,0,0}$z_0$}%
}}}}
\end{picture}%